\documentclass[11pt,a4paper]{article}
\usepackage[T1]{fontenc}
\usepackage{lmodern}
\usepackage[a4paper,margin=25mm]{geometry}
\usepackage{amsmath,amssymb,amsthm,mathtools}
\usepackage{microtype}
\usepackage{enumitem}
\usepackage[authoryear,round]{natbib}
\usepackage[hidelinks]{hyperref}
\hypersetup{pdftitle={Bergsma--Dassios sign covariance with ties: a nonnegative decomposition and sharp bounds},
pdfauthor={Wicher Bergsma and Angelos Dassios}}
\theoremstyle{plain}
\newtheorem{theorem}{Theorem}
\newtheorem{lemma}[theorem]{Lemma}
\newtheorem{proposition}[theorem]{Proposition}
\newtheorem{corollary}[theorem]{Corollary}
\newtheorem*{atomlessidentification}{Atomless identification}
\theoremstyle{definition}
\newtheorem{remark}[theorem]{Remark}

\DeclareMathOperator{\sgn}{sign}
\newcommand{\R}{\mathbb R}
\newcommand{\E}{\mathbb E}
\newcommand{\1}{\mathbf 1}
\newcommand{\ts}{\tau^{*}}
\newcommand{\Ra}{\widetilde R}
\newcommand{\Da}{\widetilde D}
\newcommand{\indep}{\perp\!\!\!\perp}
\newcommand{\cc}{\mathsf c}
\newcommand{\ggap}{\mathsf g}
\newcommand{\Xc}{\mathcal X}

\title{Bergsma--Dassios sign covariance with ties:
a nonnegative decomposition and sharp bounds}
\author{Wicher Bergsma \and Angelos Dassios}
\date{}
\begin{document}
\maketitle
\begin{center}
London School of Economics and Political Science\\
Houghton Street, London WC2A 2AE, United Kingdom\\[3pt]
\small\href{mailto:w.p.bergsma@lse.ac.uk}{\texttt{w.p.bergsma@lse.ac.uk}}\quad
\href{mailto:a.dassios@lse.ac.uk}{\texttt{a.dassios@lse.ac.uk}}
\end{center}

\begin{abstract}
We construct tie-symmetrised extensions $\Da$ and $\Ra$ of the unnormalised
Hoeffding and Blum--Kiefer--Rosenblatt functionals $D$ and $R$ and prove, for
every real-valued bivariate law, the exact nonnegative decomposition
\[
 \ts(X,Y)=12\Da(X,Y)+24\Ra(X,Y),
 \qquad \Da,\Ra\geq0,
 \qquad \Ra(X,Y)=0\quad\Longleftrightarrow\quad X\indep Y.
\]
For atomless margins the components reduce to $D$ and $R$, whereas in the
presence of ties the corresponding classical identity $\ts=12D+24R$ can
fail.  The components arise from the Hoeffding projections of a symmetrised
pair kernel.  For the four
strict/non-strict boundary versions $D_{\epsilon,\eta}$ and
$R_{\epsilon,\eta}$, we establish the sharp comparisons
\[
 9\Da\geq\sum_{\epsilon,\eta}D_{\epsilon,\eta},
 \qquad
 9\Ra\geq\sum_{\epsilon,\eta}R_{\epsilon,\eta}.
\]
Consequently, $\ts\geq0$ for every bivariate law, and the
Bergsma--Dassios conjecture is settled:
$\ts=0$ if and only if $X\indep Y$.  The boundary comparisons also yield the
universal bound $\ts\geq(8/3)R$.  Finite-table arguments, including an exact
sum-of-squares certificate, establish the boundary inequalities, and
weak-order quantisation transfers them to arbitrary laws.  The factor $9$ in
each comparison and the coefficient $24$ in $\ts\geq24\Ra$ are sharp.  The
usual permutation test is consistent against every fixed dependent
alternative, without assumptions on the margins.
\end{abstract}

\noindent\textbf{MSC2020 subject classifications:} Primary 62H20; secondary 62H05, 62G10.


\section{Introduction and main results}

Rank-based procedures for testing independence must account for ties when
variables are ordinal, counts, rounded measurements, or have mixed marginal
types.  Pairs of such variables can have joint laws with mixed-dimensional
support.  For instance,
an ordinal variable paired with a conditionally continuous measurement is
typically supported on parallel lines, while coordinatewise censoring of a
continuous pair can create a point atom, non-atomic boundary-line components,
and an absolutely continuous interior.  Such observed-data laws need be
neither purely atomic nor absolutely continuous with respect to
two-dimensional Lebesgue measure, nor mixtures of those two types, and their
margins need not both be atomless.  A tie-aware theory therefore covers
standard mixed measurement settings, not only exceptional singular laws; see
also \citet{GNRM} on independence testing with arbitrary margins.

For $z=(z_1,z_2,z_3,z_4)\in\R^4$, let
$S_{ij\mid kl}(z)=|z_i-z_j|+|z_k-z_l|$ denote the matching sum for
$ij\mid kl$, and put
\[
 a(z)=\sgn\{S_{12\mid34}(z)-S_{13\mid24}(z)\},
\]
where $\sgn(0)=0$.  Let $(X_i,Y_i)$, $i\geq1$, be independent copies of
$(X,Y)$.  The sign covariance $\ts$, introduced together with its associated
independence test in \citet{BD}, is
\begin{equation}\label{eq:tau-def}
 \ts(X,Y)=
 \E\{a(X_1,X_2,X_3,X_4)a(Y_1,Y_2,Y_3,Y_4)\}.
\end{equation}
The sample version $t_n^*$ is a bounded order-four rank $U$-statistic.  Its
large-sample theory is developed in \citet{NWD}.

To describe the structural issue created by ties, write
$F(x,y)=\Pr(X\leq x,Y\leq y)$ and let $F_X,F_Y$ be the marginal distribution
functions.  Let $D$ be Hoeffding's unnormalised functional~\citep{Hoeffding}
and $R$ the unnormalised Blum--Kiefer--Rosenblatt functional~\citep{BKR}:
\begin{align*}
 D(X,Y)&=\int_{\R^2}\{F(x,y)-F_X(x)F_Y(y)\}^2\,dF(x,y),\\
 R(X,Y)&=\int_{\R^2}\{F(x,y)-F_X(x)F_Y(y)\}^2
           \,dF_X(x)\,dF_Y(y).
\end{align*}
This notation follows \citet{DHS} and \citet{SDH}; \citet{BD} instead write
$H$ for our $D$ and $D$ for our $R$.  For absolutely continuous laws, the
combination $D+2R$ already appears in
\citet[proof of Prop.~9, p.~62]{Yanagimoto}.  For atomless margins,
Equation~(6.1) of \citet[p.~3225]{DHS} is the kernel identity underlying the
decomposition stated explicitly in \citet[Rem.~1, p.~322]{SDH}:
\begin{equation}\label{eq:decomp1}
 \ts=12D+24R.
\end{equation}
The identity in \citet{DHS} is pointwise on configurations without coordinate
ties and can fail otherwise, even for the tied $3\times3$ probability table
given below.  Although $D$ and $R$ themselves are defined for arbitrary
laws---and $R$ is nonnegative and vanishes exactly under independence
\citep[Thm.~2]{WDM}---they therefore do not decompose $\ts$ in general.

Our primary contribution is an exact replacement that remains valid with
atoms, ties, and singular components.  We group the four arguments of the
$a$-kernel into two pairs and use the resulting symmetric pair kernel to
define tie-symmetrised components $\Da$ and $\Ra$.  Relative to this natural
pair symmetrisation, its Hoeffding projections uniquely determine $\Ra$ and
$\Da$ as the first-order and fully degenerate second-order cross-components.
For every bivariate law,
\begin{equation}\label{eq:intro-decomposition}
 \ts=12\Da+24\Ra,
 \qquad \Da,\Ra\geq0,
 \qquad \Ra=0\ \Longleftrightarrow\ X\indep Y.
\end{equation}
For atomless margins, $\Da=D$ and $\Ra=R$.  Thus the unrestricted
independence characterization follows from a stronger structural result.
The related conditional-covariance decomposition for Kendall's tau in
Subsection~\ref{subsec:monotone-association} also shows that projection
orthogonality alone does not imply componentwise positivity.

\citet{BD} proved that $\ts\geq0$, with equality if and only if
$X\indep Y$, when the joint law is discrete, absolutely continuous, or a
mixture of the two, and conjectured the conclusion for every bivariate law
\citep[Thm.~1 and the paragraph following it, p.~1010]{BD}.
\citet[Thm.~6.1]{DHS} subsequently proved it under the assumption of atomless
margins.  A recent preprint by \citet{GrunewaldHuang} gives a different proof
of the unrestricted characterization using labelled path trees and quartet
covariance, and obtains the bound $\ts\geq2R$.  The distinguishing object here
is the universal nonnegative decomposition
\eqref{eq:intro-decomposition} and the quantitative theory it supports.

In particular, we compare $\Da$ and $\Ra$ with all four strict/non-strict
boundary versions $D_{\epsilon,\eta}$ and $R_{\epsilon,\eta}$ defined in
\eqref{eq:D-boundary}--\eqref{eq:R-boundary}.  We prove the sharp inequalities
\[
 9\Da\geq\sum_{\epsilon,\eta\in\{<,\leq\}}D_{\epsilon,\eta},
 \qquad
 9\Ra\geq\sum_{\epsilon,\eta\in\{<,\leq\}}R_{\epsilon,\eta}.
\]
Since $R_{\leq,\leq}=R$, they yield the stronger universal comparison
\[
 \ts\geq24\Ra
 \geq\frac83\sum_{\epsilon,\eta\in\{<,\leq\}}R_{\epsilon,\eta}
 \geq\frac83R.
\]
The factor $9$ in each boundary comparison and the coefficient $24$ relative
to $\Ra$ are best possible; optimality of the displayed coefficient $8/3$
relative to $R$ is not asserted.  Finite-table arguments, including an exact
sum-of-squares certificate, establish the inequalities, which are then
transferred to arbitrary laws by weak-order quantisation.  We also obtain sharp
quadratic bounds involving Kendall's tau and Spearman's rho, tie-aware range
results, and a projection-averaged independence characterization for random
vectors.

The universal zero-set conclusion also makes the usual permutation test based
on $t_n^*$ consistent against every fixed dependent alternative, without
continuity or moment assumptions.  Finite-sample validity and consistency are
stated after Theorem~\ref{thm:main} and proved in Section~2 of the
Supplementary Material~\citep{Supplement}.

The sign covariance belongs to the general framework of symmetric rank
covariances developed in \citet{WDM}.  In the atomless setting, its
eight-pattern characterization has also been studied through permutons and
quasirandom permutations \citep{Chan}.  A distinct quasirandomness-forcing
linear combination of six patterns of lengths three and four is given by
\citet{CrudeleDukesNoel}; it likewise determines independence for atomless
bivariate laws.  Neither permuton approach treats ties.  Pattern-based tests
for two-dimensional copulas are developed by \citet{BaringhausGrubel2025},
and recent efficiency results are given by \citet{BaringhausGrubel}; see
\citet{SDH} for a comparison of $D$, $R$ and $\ts$ that discusses ties.  The
present paper instead develops an exact tie-aware decomposition of the
original sign covariance, together with positivity, sharp bounds, and its
zero set.

\medskip
\noindent\emph{Normalisation.}
The coefficients $12$ and $24$ use the unnormalised convention for $D$ and
$R$ of \citet{DHS} and \citet{SDH}.  In the source-kernel normalisation of
\citet{DHS}, the kernels $h_D$, $h_R$ and $h_{\tau^*}$ satisfy
$\E h_D=30D$, $\E h_R=90R$ and $\E h_{\tau^*}=\tfrac32\ts$.  Taking
expectations in their identity gives
\[
 \E h_{\tau^*}=\frac35\E h_D+\frac25\E h_R,
\]
which is equivalent to~\eqref{eq:decomp1}.  Unit-maximum scales and the effect
of ties are considered in
Subsection~\ref{subsec:ranges-normalisation}.

\subsection{Tie-symmetrised components and decomposition}

We first define
\begin{equation}\label{eq:Ra-def}
 \boxed{\qquad
 \Ra(X,Y)=\frac19\E\{a(X_1,X_2,X_3,X_4)
          a(Y_1,Y_2,Y_5,Y_6)\}.
 \qquad}
\end{equation}
For structural purposes, let $U,V$ be independent copies of a real random
variable $Z$ and introduce the auxiliary kernel
\[
 k_Z(z,z')=-\frac13\E\{a(z,z',U,V)\}
 =\frac13\E\{a(z,U,z',V)\}.
\]
The equality follows from antisymmetry under interchange of the second and
third slots; the same antisymmetry and exchangeability show that $k_Z$ is
centred: $\E\{k_Z(Z_1,Z_2)\}=0$.  The sign and normalisation are chosen so
that, in the atomless case, $k_Z$ agrees
$\mathcal L(Z)^{\otimes2}$-almost everywhere with the centred
Cram\'er--von Mises kernel; see \eqref{eq:k-atomless}.
Conditioning \eqref{eq:Ra-def} on the first two joint observations gives, for
every bivariate law,
\[
 \Ra(X,Y)=\E\{k_X(X_1,X_2)k_Y(Y_1,Y_2)\}.
\]
The following statement shows that $\Ra$ is an $a$-kernel analogue of $R$ that
agrees with it for atomless marginals.

\begin{atomlessidentification}
If both marginal laws of $(X,Y)$ are atomless, then $\Ra(X,Y)=R(X,Y)$.
\end{atomlessidentification}

\begin{proof}
Define the centred Cram\'er--von Mises kernel
\[
 \kappa_Z(z,z')=\int_{\R}
 \{\1(z\leq t)-F_Z(t)\}\{\1(z'\leq t)-F_Z(t)\}\,dF_Z(t).
\]
If $F_Z$ is continuous, a probability-integral-transform calculation,
detailed in Section~5 of the Supplementary Material~\citep{Supplement}, gives, for
$\mathcal L(Z)^{\otimes2}$-almost every $(z,z')$,
\begin{equation}\label{eq:k-atomless}
 k_Z(z,z')=\frac{F_Z(z)^2+F_Z(z')^2}{2}
 -\max\{F_Z(z),F_Z(z')\}+\frac13=\kappa_Z(z,z').
\end{equation}
Consequently, Fubini's theorem and
independence of copies give
\begin{align*}
 \Ra(X,Y)
 &=\E\{\kappa_X(X_1,X_2)\kappa_Y(Y_1,Y_2)\}\\
 &=\int_{\R^2}\!\left[\E\{(\1(X\leq x)-F_X(x))
       (\1(Y\leq y)-F_Y(y))\}\right]^2dF_X(x)dF_Y(y)\\
 &=R(X,Y).\qedhere
\end{align*}
\end{proof}

We next construct $\Da$, the other tie-symmetrised term in the decomposition of
$\ts$.  For $Z=X$ or $Y$ and distinct indices $i,j,k,l$, write
\[
 S^Z_{ij\mid kl}=|Z_i-Z_j|+|Z_k-Z_l|,
 \qquad
 a_Z(ijkl)=\sgn(S^Z_{ij\mid kl}-S^Z_{ik\mid jl}),
\]
and define the symmetric pair contrast
\[
 \delta_Z(ijkl)=a_Z(ijkl)-a_Z(iljk)
 =\sgn(S^Z_{ij\mid kl}-S^Z_{ik\mid jl})
   +\sgn(S^Z_{ij\mid kl}-S^Z_{il\mid jk}).
\]
Subsection~\ref{subsec:projection-structure} verifies
\begin{equation}\label{eq:delta-products}
 \E\{\delta_X(1234)\delta_Y(1234)\}=3\ts,
 \qquad
 \E\{\delta_Z(1234)\mid Z_1,Z_2\}=-6k_Z(Z_1,Z_2).
\end{equation}

Define the normalised degenerate projection by
\begin{equation}\label{eq:ell-def}
 \ell_Z(1234)=\frac16\delta_Z(1234)
   +k_Z(Z_1,Z_2)+k_Z(Z_3,Z_4).
\end{equation}
For arbitrary marginal laws, define
\begin{equation}\label{eq:Da}
 \boxed{\qquad
 \Da(X,Y)=\E\{\ell_X(1234)\ell_Y(1234)\}.
 \qquad}
\end{equation}

For later use, define the four-block rectangular contrast
\begin{equation}\label{eq:K-rectangular}
 H_Z=\frac1{12}\{
   \delta_Z(1234)-\delta_Z(1256)
   -\delta_Z(3478)+\delta_Z(5678)\}.
\end{equation}
The projection calculation below proves the equivalent eight-copy
representation
\begin{equation}\label{eq:Da-rectangular}
 \Da(X,Y)=\E\{H_XH_Y\}.
\end{equation}
The same calculation gives the central identity
\begin{equation}\label{eq:tau-decomposition}
 \boxed{\qquad \ts(X,Y)=12\Da(X,Y)+24\Ra(X,Y).\qquad}
\end{equation}

\begin{atomlessidentification}
If both marginal laws of $(X,Y)$ are atomless, then $\Da(X,Y)=D(X,Y)$.
\end{atomlessidentification}

\begin{proof}
The earlier atomless identification of $\Ra$ gives $\Ra=R$.  Comparing
\eqref{eq:tau-decomposition} with the classical atomless decomposition
\eqref{eq:decomp1} therefore gives $\Da=D$.
\end{proof}

\smallskip
\noindent\emph{A tied $3\times3$ example.}
Let $(X,Y)$ have probability table
\[
 P=\frac13\begin{pmatrix}1&0&0\\0&1&0\\0&0&1\end{pmatrix}.
\]
Direct calculation gives
\[
 \ts=\frac{28}{81},\qquad D=\frac8{243},\qquad R=\frac{10}{729},
 \qquad 12D+24R=\frac{176}{243}\ne\ts,
\]
whereas
\[
 \Da=\frac{73}{6561},\qquad \Ra=\frac{58}{6561},
 \qquad 12\Da+24\Ra=\frac{28}{81}=\ts.
\]
Thus the classical atomless decomposition can fail under ties, while the
tie-symmetrised decomposition remains exact.

Exchangeability and the centring of $k_Z$ give
$\E\{a_Z(1234)\}=\E\{\ell_Z(1234)\}=0$.  Thus $\ts=\Ra=\Da=0$ under
independence, whereas nonnegativity of either $\Ra$ or $\Da$ is not immediate
from their definitions.

\subsection{Main results}

\begin{theorem}\label{thm:main}
For every bivariate law on $\R^2$:
\begin{enumerate}[label=\textup{(\roman*)},leftmargin=2.2em,itemsep=2pt,topsep=2pt]
\item $\Da(X,Y)\geq0$, $\Ra(X,Y)\geq0$, and the exact decomposition
\eqref{eq:tau-decomposition} holds.  Consequently, $\ts(X,Y)\geq0$.
\item Under independence, $\Da=\Ra=0$. Moreover, 
\[
 \boxed{\qquad
 \Ra(X,Y)=0\quad\Longleftrightarrow\quad X\indep Y
 \quad\Longleftrightarrow\quad \ts(X,Y)=0.
 \qquad}
\]
\item If both marginal laws are atomless, then $\Da=D$ and $\Ra=R$, in the
unnormalised conventions used here.
\item The equivalence $\Da(X,Y)=0\Longleftrightarrow X\indep Y$ holds when the
joint law is either purely atomic or absolutely continuous with respect to
two-dimensional Lebesgue measure.  In general it can fail, even when both
margins are uniform.
\item Every dependent purely atomic law satisfies $\ts>24\Ra$.  Nevertheless,
$24$ is the best constant in the universal bound $\ts\geq24\Ra$: there is a
sequence of dependent finite probability tables along which $\ts/\Ra\to24$.
\end{enumerate}
\end{theorem}

The general $\Da$ counterexample in Theorem~\ref{thm:main} is Yanagimoto's
law, recalled in the sharpness argument below.

\medskip
\noindent\emph{Statistical consequence.}
For every $n\geq4$, the usual nonrandomised upper-tail permutation test based
on $t_n^*$ has level at most $\alpha$ under independence.  Under every fixed
dependent alternative, its rejection probability tends to one, with no
restriction on the margins.  Finite-sample validity and consistency are
proved in Section~2 of the Supplementary Material~\citep{Supplement}.

We establish stronger quantitative comparisons with all four boundary
versions of the Hoeffding and product-margin Blum--Kiefer--Rosenblatt
functionals.  Put
\[
 \iota_{<}(u,t)=\1\{u<t\},
 \qquad
 \iota_{\leq}(u,t)=\1\{u\leq t\}.
\]
For $\epsilon,\eta\in\{<,\leq\}$, define
\[
\begin{aligned}
 F_{\epsilon,\eta}(x,y)
   &=\E\{\iota_\epsilon(X,x)\iota_\eta(Y,y)\},\quad
 F_{X,\epsilon}(x)=\E\{\iota_\epsilon(X,x)\},\qquad
 F_{Y,\eta}(y)=\E\{\iota_\eta(Y,y)\},\\
 \Delta_{\epsilon,\eta}(x,y)
   &=F_{\epsilon,\eta}(x,y)-F_{X,\epsilon}(x)F_{Y,\eta}(y),
\end{aligned}
\]
and define
\begin{align}
 D_{\epsilon,\eta}(X,Y)
 &=\E\{\Delta_{\epsilon,\eta}(X,Y)^2\},
 \label{eq:D-boundary}\\
 R_{\epsilon,\eta}(X,Y)
 &=\int_{\R^2}\Delta_{\epsilon,\eta}(x,y)^2
       \,P_X(dx)P_Y(dy).
 \label{eq:R-boundary}
\end{align}
In particular, $D_{\leq,\leq}=D$ and $R_{\leq,\leq}=R$ for every law.  All
four $D$-functionals agree with $D$, and all four $R$-functionals agree with
$R$, when both marginals are atomless.

\begin{theorem}[Four-boundary strengthening]\label{thm:four-boundary}
For every bivariate law on $\R^2$, with both sums below taken over
$\epsilon,\eta\in\{<,\leq\}$,
\begin{equation}\label{eq:four-boundary-main}
 \boxed{\quad
 9\Da(X,Y)\geq\sum_{\epsilon,\eta}D_{\epsilon,\eta}(X,Y),\qquad
 9\Ra(X,Y)\geq\sum_{\epsilon,\eta}R_{\epsilon,\eta}(X,Y).
 \quad}
\end{equation}
Consequently,
\[
 \boxed{\qquad
 \ts(X,Y)\geq24\Ra(X,Y)+\frac43
 \sum_{\epsilon,\eta\in\{<,\leq\}}D_{\epsilon,\eta}(X,Y)
 \geq24\Ra(X,Y)\geq0.
 \qquad}
\]
The constant $9$ in each inequality in \eqref{eq:four-boundary-main} is best
possible.
\end{theorem}

\subsection{Projection structure and proof of the decomposition}
\label{subsec:projection-structure}

\begin{proof}[Verification of \eqref{eq:delta-products}]
For a fixed quadruple $z$, write $ij\parallel kl$ when the pairs $(z_i,z_j)$
and $(z_k,z_l)$ are strictly separated, that is, when
$\max\{z_i,z_j\}<\min\{z_k,z_l\}$ or vice versa.  Direct inspection gives the
exhaustive case table
\[
\begin{array}{c|c@{\qquad}c@{\qquad}c@{\qquad}c}
 &12\parallel34&13\parallel24&14\parallel23&\text{none}\\ \hline
 a_Z(1234)&-1&1&0&0\\
 a_Z(1423)& 1&0&-1&0
\end{array}
\]
For $Z=X,Y$, let $E_1^Z$, $E_2^Z$ and $E_3^Z$ denote, respectively, the
events in the first three columns of the table.  Reading off the nonzero
products gives
\[
 \E\{a_X(1234)a_Y(1423)\}
 =-\Pr(E_1^X\cap E_1^Y)+\Pr(E_1^X\cap E_3^Y)
  +\Pr(E_2^X\cap E_1^Y)-\Pr(E_2^X\cap E_3^Y).
\]
By exchangeability, $p_d=\Pr(E_i^X\cap E_i^Y)$ does not depend on $i$, and
$p_o=\Pr(E_i^X\cap E_j^Y)$ does not depend on distinct $i,j$.  Therefore the
last display equals $p_o-p_d$, whereas the same table gives
$\ts=2(p_d-p_o)$.  Hence
\[
 \E\{a_X(1234)a_Y(1423)\}=-\frac{\ts}{2}.
\]
Expanding the two contrasts and using exchangeability once more gives
\[
 \E\{\delta_X(1234)\delta_Y(1234)\}
 =2\ts-2(-\ts/2)=3\ts.
\]
This proves the first identity in \eqref{eq:delta-products}.
The conditional identity in \eqref{eq:delta-products} follows immediately
from the two representations of $k_Z$.
\end{proof}

The decomposition now follows from an elementary pair-kernel projection
identity, which also yields the rectangular representation.

\begin{lemma}[pair-kernel projection identity]\label{lem:pair-projection}
Let $A_1,A_2,A_3,A_4$ be independent and identically distributed random
elements, and let $f$ and $g$ be square-integrable, symmetric, mean-zero
kernels of two arguments.
Define
\[
 f_1(a)=\E\{f(a,A_2)\},
 \qquad
 f_2(a,b)=f(a,b)-f_1(a)-f_1(b),
\]
and define $g_1,g_2$ analogously.  Then
\begin{equation}\label{eq:proj-ident1}
 \E\{f(A_1,A_2)g(A_1,A_2)\}
 =2\E\{f_1(A_1)g_1(A_1)\}
  +\E\{f_2(A_1,A_2)g_2(A_1,A_2)\}.
\end{equation}
\begin{samepage}
Moreover, if
\begin{equation}\label{eq:rect-contrast}
   \mathcal R_f=f(A_1,A_2)-f(A_1,A_3)-f(A_2,A_4)+f(A_3,A_4),
\end{equation}
and $\mathcal R_g$ is defined analogously, then
\[
 \E(\mathcal R_f\mathcal R_g)
 =4\E\{f_2(A_1,A_2)g_2(A_1,A_2)\}.
\]
The decompositions into first-order terms and a fully degenerate second-order
term are unique up to almost-sure equality.
\end{samepage}
\end{lemma}
\pagebreak[3]

\begin{proof}
Put
$u_i=f_1(A_i)$, $u_{ij}=f_2(A_i,A_j)$, $v_i=g_1(A_i)$, and $v_{ij}=g_2(A_i,A_j)$.
Then $f(A_1,A_2)=u_1+u_2+u_{12}$ and
$g(A_1,A_2)=v_1+v_2+v_{12}$.  Conditional centring and independence make all
terms in the expansion vanish except
$u_1v_1$, $u_2v_2$ and $u_{12}v_{12}$.  The first two have the same
expectation, proving~\eqref{eq:proj-ident1}. 
In~\eqref{eq:rect-contrast} the first-order terms cancel, leaving
\[
 \mathcal R_f=u_{12}-u_{13}-u_{24}+u_{34},
 \qquad
 \mathcal R_g=v_{12}-v_{13}-v_{24}+v_{34}.
\]
For distinct edges $\{i,j\}\ne\{k,l\}$, $\E(u_{ij}v_{kl})=0$: condition on
the common block when the edges meet, and use independence and centring when
they are disjoint.  The four matching-edge products have the same expectation,
which proves the second identity.  Finally, uniqueness follows by conditioning:
if $f(A_1,A_2)=h(A_1)+h(A_2)+w(A_1,A_2)$, with $\E h(A_1)=0$ and $w$
conditionally centred, then $h(A_1)=f_1(A_1)$ almost surely and hence
$w=f_2$; similarly for $g$.
\end{proof}

\begin{proof}[Proof of decomposition \eqref{eq:tau-decomposition} and representation
\eqref{eq:Da-rectangular}]
For $i=1,\ldots,4$, group the observations into the independent blocks
\[
 A_i=((X_{2i-1},Y_{2i-1}),(X_{2i},Y_{2i})),
\]
and set
\[
 f(A_1,A_2)=\delta_X(1234),
 \qquad
 g(A_1,A_2)=\delta_Y(1234).
\]
For $Z=X,Y$, centring of $k_Z$ and the conditional identity in
\eqref{eq:delta-products} give $\E\{\delta_Z(1234)\}=0$, while the sign-sum
representation shows that $\delta_Z(1234)$ is symmetric in the two blocks
$(Z_1,Z_2)$ and $(Z_3,Z_4)$.  Thus $f$ and $g$ are symmetric, mean-zero
kernels, so Lemma~\ref{lem:pair-projection} applies.  By
\eqref{eq:delta-products} and \eqref{eq:ell-def},
\[
\begin{aligned}
 f_1(A_1)&=-6k_X(X_1,X_2),&
 f_2(A_1,A_2)&=6\ell_X(1234),\\
 g_1(A_1)&=-6k_Y(Y_1,Y_2),&
 g_2(A_1,A_2)&=6\ell_Y(1234).
\end{aligned}
\]
Thus $\ell_Z$ is fully degenerate, and the product representation of $\Ra$
above, together with \eqref{eq:Da}, gives
\[
 \E(f_1g_1)=36\Ra,
 \qquad
 \E(f_2g_2)=36\Da.
\]
Together with the first identity in \eqref{eq:delta-products},
\eqref{eq:proj-ident1} therefore gives
\[
 3\ts=\E(fg)=2\E(f_1g_1)+\E(f_2g_2)=72\Ra+36\Da,
\]
which is \eqref{eq:tau-decomposition}.

Finally, \eqref{eq:rect-contrast} and \eqref{eq:K-rectangular} give
$\mathcal R_f=12H_X$ and $\mathcal R_g=12H_Y$.  The rectangular identity in
Lemma~\ref{lem:pair-projection} yields
\[
 \E(H_XH_Y)
 =\frac1{144}\E(\mathcal R_f\mathcal R_g)
 =\frac1{36}\E(f_2g_2)=\Da,
\]
proving \eqref{eq:Da-rectangular}.
\end{proof}

Thus the decomposition is canonical relative to the pair-symmetrised kernel
$\delta_Z$ and the pair-block structure; its special feature is the
nonnegativity of both cross-components proved below.

\medskip
\noindent\emph{Proof strategy.}
Section~\ref{sec:cut-gap} develops the common cut--gap algebra and proves the
finite-table $\Ra$ bound and its equality case.  Section~\ref{sec:finite-certificate}
then derives the $\Da$ gap correction and proves its finite-table bound through
an exact nonnegative certificate.
Section~\ref{sec:arbitrary-laws} transfers the result to arbitrary
laws by weak-order quantisation and proves sharpness.  Theorem~\ref{thm:four-boundary}
then gives nonnegativity; since $R_{\leq,\leq}=R$ and $R=0$ exactly under
independence, it also gives the zero set of $\Ra$, and
\eqref{eq:tau-decomposition} completes Theorem~\ref{thm:main}.

\section{Cut--gap algebra and the finite \texorpdfstring{$\Ra$}{R-tilde} boundary bound}
\label{sec:cut-gap}

Throughout this section, $P=(p_{ij})$ denotes a finite probability table whose
attained row and column levels are ordered as $1,\ldots,m$ and $1,\ldots,n$.
We identify $P$ with the corresponding law of $(X,Y)$ and write $\ts(P)$,
$\Ra(P)$ and $\Da(P)$ for the corresponding functional values.

The main finite-table result for $\Ra$ is the following.

\begin{proposition}[finite $\Ra$ boundary inequality]\label{prop:finite-R}
For every such $P$,
\begin{equation}\label{eq:finite-R-boundary}
\boxed{\quad
 9\Ra(P)\geq
 \sum_{\epsilon,\eta\in\{<,\leq\}}R_{\epsilon,\eta}(P)\geq0.
\quad}
\end{equation}
\end{proposition}

We prove the proposition by first establishing the signed cut--gap
representation and the cut-to-gap transform used in both finite-table
arguments, and then deriving an exact quadratic form and a one-dimensional
energy bound.  The $\Da$ argument is deferred to
Section~\ref{sec:finite-certificate}.

For $u,v\in\R$, let
$I(u,v)=\{t\in\R:\min(u,v)<t<\max(u,v)\}$ be the open interval between
$u$ and $v$; thus $I(3,2)=(2,3)$ and $I(u,u)=\varnothing$.  On the ordered
set $\{1,\ldots,d\}$, let $\cc_k$ denote the cut between $k$ and $k+1$
and let $\ggap_k$ denote the gap formed by the adjacent cuts $\cc_k$ and
$\cc_{k+1}$.  Thus $1\leq k\leq d-1$ for cuts and $1\leq k\leq d-2$ for gaps.
The cut $\cc_k$ separates $\{1,\ldots,k\}$ from $\{k+1,\ldots,d\}$,
whereas the gap $\ggap_k$ separates $\{1,\ldots,k\}$ from
$\{k+2,\ldots,d\}$, leaving level $k+1$ in neither block.
Say that $\cc_k$ lies inside $I(u,v)$ if
$u\wedge v\leq k<u\vee v$, and that $\ggap_k$ lies inside it if
$u\wedge v\leq k<k+1<u\vee v$.  For any cut or gap $\alpha$, put
\[
 \rho_\alpha(u,v)=\sgn(u-v)\1\{\alpha\text{ lies inside }I(u,v)\}.
\]

\begin{lemma}[signed cut--gap representation]\label{lem:cut-gap}
For every $z\in\R^4$,
\begin{equation}\label{eq:interval-form}
 a(z)=\sgn\{(z_1-z_4)(z_3-z_2)\}
 \1\{I(z_1,z_4)\cap I(z_2,z_3)\ne\varnothing\}.
\end{equation}
Consequently, when $z_1,\ldots,z_4\in\{1,\ldots,d\}$,
\begin{equation}\label{eq:signed-cut-gap}
 a(z)=\sum_{k=1}^{d-2}
 \rho_{\ggap_k}(z_1,z_4)\rho_{\ggap_k}(z_2,z_3)
 -\sum_{k=1}^{d-1}
 \rho_{\cc_k}(z_1,z_4)\rho_{\cc_k}(z_2,z_3).
\end{equation}
In particular, $a$ depends only on the complete weak order of its arguments,
including their ties.
\end{lemma}

\begin{proof}
 The $a_Z(1234)$ row of the case table in
 Subsection~\ref{subsec:projection-structure} lists the values of $a(z)$ for
 the four possible separation patterns.  Direct inspection shows that the
 right-hand side of \eqref{eq:interval-form} has the same values, proving the
 identity.

 If the two intervals contain $N\geq1$ common cuts, they contain $N-1$ common
 gaps.  By \eqref{eq:interval-form}, each corresponding product equals
 $-a(z)$, so the right-hand side of \eqref{eq:signed-cut-gap} evaluates to
 $-(N-1)a(z)+Na(z)=a(z)$; if $N=0$, both sides vanish.
\end{proof}

For $d\geq2$ and a nonnegative vector $w=(w_1,\ldots,w_d)$ of total mass
$s>0$, put
\[
 L_a(w)=\sum_{k\leq a}w_k,\qquad
 U_a(w)=\sum_{k\geq a+2}w_k.
\]
Define the $(d-2)\times(d-1)$ bidiagonal matrix
\[
 G(w)=
 \begin{pmatrix}
  U_1(w)&L_1(w)&&&0\\
  &\ddots&\ddots&&\\
  0&&&U_{d-2}(w)&L_{d-2}(w)
 \end{pmatrix},
\]
which will convert quantities indexed by cuts into quantities indexed by gaps.

\begin{lemma}[cut-to-gap transform]\label{lem:cut-gap-transform}
Let $d\geq2$, let $w\in\R_+^d$ have total mass $s>0$, and let $Z_w$ have
probability mass function $w/s$.  For a cut or gap $\alpha$, put
\(\pi_\alpha^w(i)=\E\rho_\alpha(i,Z_w)\), $1\leq i\leq d$.  Then, for every
$i$,
\[
 \begin{aligned}
 s(\pi_{\ggap_a}^w(i))_{1\leq a\leq d-2}
 &=G(w)(\pi_{\cc_b}^w(i))_{1\leq b\leq d-1},\\
 \text{that is,}\qquad
 s\pi_{\ggap_a}^w(i)
 &=U_a(w)\pi_{\cc_a}^w(i)+L_a(w)\pi_{\cc_{a+1}}^w(i).
 \end{aligned}
\]

Moreover, let a nonnegative $2\times d$ table of total mass $s$ have columns
$q_1,\ldots,q_d$ and column margins $w_k=q_{1k}+q_{2k}$.  Define its
column-cut and column-gap determinants by
\[
 \begin{gathered}
 d_b=\det\left(\sum_{k\leq b}q_k,\sum_{k>b}q_k\right),\qquad
 e_a=\det\left(\sum_{k\leq a}q_k,\sum_{k\geq a+2}q_k\right),\\
 s(e_a)_{1\leq a\leq d-2}
 =G(w)(d_b)_{1\leq b\leq d-1}.
 \end{gathered}
\]
\end{lemma}

\begin{proof}
Directly, $\pi_{\cc_a}^w(i)=L_a(w)/s-\1\{i\leq a\}$ and
$s\pi_{\ggap_a}^w(i)=L_a(w)\1\{i\geq a+2\}
-U_a(w)\1\{i\leq a\}$.
Since $L_{a+1}(w)=L_a(w)+w_{a+1}$ and
$s=L_a(w)+w_{a+1}+U_a(w)$, these formulas give the score identity.

For the determinant form, fix $a$ and put
$A=\sum_{k\leq a}q_k$, $V=q_{a+1}$, and
$C=\sum_{k\geq a+2}q_k$.  The total masses of $A,V,C$ are
$L_a(w),w_{a+1},U_a(w)$, respectively.  Put
$M=\left(\begin{smallmatrix}L_a(w)&w_{a+1}&U_a(w)\\
A_1&V_1&C_1\\ A_2&V_2&C_2\end{smallmatrix}\right)$.
Its first row is the sum of the other two, so expanding $\det M=0$ along that
row gives
$U_a(w)\det(A,V)+L_a(w)\det(V,C)=w_{a+1}\det(A,C)$.
Since $d_a=\det(A,V+C)$, $d_{a+1}=\det(A+V,C)$ and
$e_a=\det(A,C)$,
bilinearity gives
\[
 U_a(w)d_a+L_a(w)d_{a+1}
 =\{w_{a+1}+U_a(w)+L_a(w)\}e_a=se_a,
\]
which proves the determinant identity.
\end{proof}

Thus $G(w)/s$ is the common cut-to-gap operator in the score and determinant
representations.  Define the associated difference-of-energies matrix by
\[
 K(w)=s^2I_{d-1}-G(w)^{\mathsf T}G(w).
\]
When $d=2$, $G(w)$ has no rows and $K(w)=s^2I_1$.

For the proof, write
\[
 r_i=\sum_jp_{ij},\qquad c_j=\sum_ip_{ij},
\]
and, for $1\leq a\leq m-1$ and $1\leq b\leq n-1$, put
\[
 x_{ab}=P(X\leq a,Y\leq b)-P(X\leq a)P(Y\leq b).
\]
Let $\Xc=(x_{ab})_{1\leq a\leq m-1,\,1\leq b\leq n-1}$.
Extend \(x_{ab}\) to the boundary by setting $x_{0b}=x_{mb}=x_{a0}=x_{an}=0$.  At a threshold pair $(i,j)$, the four
boundary residuals are then
\[
 (\Delta_{\leq,\leq},\Delta_{<,\leq},
 \Delta_{\leq,<},\Delta_{<,<})(i,j)
 =(x_{ij},x_{i-1,j},x_{i,j-1},x_{i-1,j-1}).
\]

Combining the six-copy definition \eqref{eq:Ra-def} with the signed cut--gap
representation \eqref{eq:signed-cut-gap} gives the following finite-table
quadratic form.

\begin{corollary}[finite $\Ra$ quadratic form]
For the finite probability table $P$ above, assume $m,n\geq2$.  Then
\begin{equation}\label{eq:master-Ra}
 9\Ra(P)
 =\|\Xc\|_F^2-\|\Xc G(c)^{\mathsf T}\|_F^2-\|G(r)\Xc\|_F^2
   +\|G(r)\Xc G(c)^{\mathsf T}\|_F^2
   =\operatorname{tr}\{K(r)\Xc K(c)\Xc^{\mathsf T}\}.
\end{equation}
\end{corollary}

\begin{proof}
For independent marginal draws $X'\sim r$ and
$Y'\sim c$, define
\[
 \pi^X_\alpha(i)=\E\rho_\alpha(i,X'),\qquad
 \pi^Y_\beta(j)=\E\rho_\beta(j,Y').
\]
Let $\mathcal I_r$ comprise the row cuts and gaps, and $\mathcal I_c$ the
column cuts and gaps.  Put $\sigma_\alpha=-1$ for a cut and
$\sigma_\alpha=1$ for a gap, and define
$C_{\alpha\beta}=\sum_{i,j}p_{ij}\pi^X_\alpha(i)\pi^Y_\beta(j)$.
Apply \eqref{eq:signed-cut-gap} to the two sign kernels in
\eqref{eq:Ra-def}.  The factors indexed by $\{1,4,6\}$ and $\{2,3,5\}$ are
independent and identically distributed, so the expansion factors as
\[
 9\Ra(P)=\sum_{\alpha\in\mathcal I_r}\sum_{\beta\in\mathcal I_c}
 \sigma_\alpha\sigma_\beta C_{\alpha\beta}^2.
\]
For two cuts, the coefficient is
\[
 \sum_{i,j}p_{ij}\pi^X_{\cc_a}(i)\pi^Y_{\cc_b}(j)
 =\E[(L_a(r)-\1\{X\leq a\})(L_b(c)-\1\{Y\leq b\})]
 =x_{ab}.
\]
Since $r$ and $c$ have mass one, the component identity in
Lemma~\ref{lem:cut-gap-transform} gives, for example,
\[
 C_{\cc_a\ggap_b}
 =U_b(c)x_{ab}+L_b(c)x_{a,b+1}
 =(\Xc G(c)^{\mathsf T})_{ab}.
\]
Applying the same score transform in the row coordinate and then in both
coordinates, the cut--gap, gap--cut and gap--gap coefficient matrices are
therefore
\[
 \Xc G(c)^{\mathsf T},\qquad G(r)\Xc,\qquad
 G(r)\Xc G(c)^{\mathsf T}.
\]
Substituting these four coefficient blocks into the signed sum obtained from
\eqref{eq:signed-cut-gap} gives the first equality in
\eqref{eq:master-Ra}.  Since $r$ and $c$ have mass one,
$K(r)=I-G(r)^{\mathsf T}G(r)$ and
$K(c)=I-G(c)^{\mathsf T}G(c)$.  Expanding the trace gives the second equality.
\end{proof}

The remaining ingredient is the following one-dimensional energy identity.

\begin{lemma}[one-dimensional energy]\label{lem:one-dimensional}
Let $d\geq2$, and let $w\in\R_+^d$ have mass $s>0$.  For every
$u=(u_1,\ldots,u_{d-1})^{\mathsf T}$,
\begin{equation}\label{eq:weighted-squares}
\begin{split}
 \frac1s u^{\mathsf T}K(w)u
 ={}&\sum_{b=1}^{d-1}(w_b+w_{b+1})u_b^2\\
 &+\frac1s\sum_{b=1}^{d-2}\left[
 L_b(w)U_b(w)(u_b-u_{b+1})^2
 +w_{b+1}\{U_b(w)u_b^2+L_b(w)u_{b+1}^2\}\right].
\end{split}
\end{equation}
In particular,
\[
 K(w)\succeq s\operatorname{diag}(w_b+w_{b+1})_{b=1}^{d-1}.
\]
\end{lemma}

\begin{proof}
Set $L_0(w)=U_{d-1}(w)=0$.  The definitions give
$s-U_b(w)-L_{b-1}(w)=w_b+w_{b+1}$ for $1\leq b\leq d-1$.
Using this identity to regroup the definition of $K(w)$ gives
\begin{align*}
 u^{\mathsf T}K(w)u
 ={}&s\sum_{b=1}^{d-1}(w_b+w_{b+1})u_b^2\\
 &+\sum_{b=1}^{d-2}\left[
 s\{U_b(w)u_b^2+L_b(w)u_{b+1}^2\}
 -\{U_b(w)u_b+L_b(w)u_{b+1}\}^2\right].
\end{align*}
For $U,L\geq0$, the identity
$(U+L)(Uu^2+Lv^2)-(Uu+Lv)^2=UL(u-v)^2$ holds.
Since $s=U_b(w)+L_b(w)+w_{b+1}$, applying this identity to each bracket and
dividing by $s$ gives \eqref{eq:weighted-squares}.  Positive
semidefiniteness follows, and for $d=2$ the second sum is empty.
\end{proof}

For $1\leq a\leq m-1$, write
$x_a=(x_{a1},\ldots,x_{a,n-1})^{\mathsf T}$ for the $a$th row of $\Xc$;
under the boundary convention, $x_0=x_m=0$.

\begin{proof}[Proof of Proposition~\ref{prop:finite-R}]
The assertion is trivial if only one row or column is attained.  Hence assume
$m,n\geq2$; by hypothesis all margins are positive.  Since
$R_{\epsilon,\eta}$ averages over the product margins, shifting indices
gives
\begin{equation}\label{eq:finite-R-boundary-sum}
 \sum_{\epsilon,\eta\in\{<,\leq\}}R_{\epsilon,\eta}(P)
 =\sum_{a=1}^{m-1}\sum_{b=1}^{n-1}
 (r_a+r_{a+1})(c_b+c_{b+1})x_{ab}^2.
\end{equation}

Lemma~\ref{lem:one-dimensional} gives $K(c)\succeq0$.  Apply its lower bound
first to the columns of $\Xc K(c)^{1/2}$ with $w=r$, and then to each row
$x_a$ with $w=c$.  Equations \eqref{eq:master-Ra} and
\eqref{eq:finite-R-boundary-sum} give
\begin{align*}
 9\Ra(P)
 &=\operatorname{tr}\{K(r)\Xc K(c)\Xc^{\mathsf T}\}
 \geq\sum_{a=1}^{m-1}(r_a+r_{a+1})x_a^{\mathsf T}K(c)x_a\\
 &\geq\sum_{a=1}^{m-1}\sum_{b=1}^{n-1}
 (r_a+r_{a+1})(c_b+c_{b+1})x_{ab}^2
 =\sum_{\epsilon,\eta\in\{<,\leq\}}R_{\epsilon,\eta}(P)\geq0.\qedhere
\end{align*}
\end{proof}

\begin{remark}[Equality in the finite $\Ra$ bound]
\label{rem:finite-R-equality}
Equality in the first inequality of \eqref{eq:finite-R-boundary} holds if and only if
$P=rc^{\mathsf T}$ or $m=n=2$.  If $m=1$ or $n=1$, then
$P=rc^{\mathsf T}$, so assume $m,n\geq2$.  For a positive probability vector
$w\in\R^d$, \eqref{eq:weighted-squares} shows that equality in
\[
 u^{\mathsf T}K(w)u\geq
 \sum_{b=1}^{d-1}(w_b+w_{b+1})u_b^2
\]
holds for every $u$ when $d=2$, whereas for $d\geq3$ it forces $u=0$;
the same bound shows that $K(w)$ is positive definite.  Thus both steps in
the proof above are equalities for every $\Xc$ when $m=n=2$.  If $m\geq3$,
equality in the first step forces $\Xc K(c)^{1/2}=0$, hence $\Xc=0$; if
$m=2<n$, equality in the second step forces $\Xc=0$.  Finally,
$\Xc=0$ is equivalent to $P=rc^{\mathsf T}$.
\end{remark}

\section{Finite tables: the \texorpdfstring{$\Da$}{D-tilde} correction and certificate}
\label{sec:finite-certificate}

Throughout this section, $P,r,c,\Xc,x_a,G,K,L_a$ and $U_a$ retain the
notation of Section~\ref{sec:cut-gap}.

Together with Proposition~\ref{prop:finite-R}, the following proposition is
the finite-table form of Theorem~\ref{thm:four-boundary}.  The argument first
derives the exact gap--gap correction, then reduces the boundary sum to
adjacent-row energies, rewrites the correction through adjacent-determinant
residuals, and assembles the resulting local terms into a three-block
nonnegative certificate.

\begin{proposition}[finite $\Da$ boundary inequality]\label{prop:finite-D}
For every such $P$,
\begin{equation}\label{eq:finite-D-boundary}
\boxed{\quad
 9\Da(P)\geq
 \sum_{\epsilon,\eta\in\{<,\leq\}}D_{\epsilon,\eta}(P).
\quad}
\end{equation}
Consequently,
\[
 \ts(P)\geq24\Ra(P)
 +\frac43\sum_{\epsilon,\eta\in\{<,\leq\}}D_{\epsilon,\eta}(P)
 \geq24\Ra(P)\geq0.
\]
\end{proposition}

We prove Proposition~\ref{prop:finite-D} using the determinant form of the
cut-to-gap transform, Lemma~\ref{lem:cut-gap-transform}, and three further
results.  Proposition~\ref{prop:master} isolates the correction term for
$\Da$, and
Lemma~\ref{lem:opposite-shift-pairing} supplies the local nonnegative pairing
used in the final certificate.  Lemma~\ref{lem:row-block-shifts} records the
sign-sensitive row-block identities that connect the table to that pairing.

\begin{proposition}[gap--gap correction formula]\label{prop:master}
Suppose $m,n\geq2$ and all row and column margins of $P$ are positive.  Put
$s_a=L_a(r)+U_a(r)=1-r_{a+1}$ and, for
$1\leq a\leq m-2$,
\[
 z_a=U_a(r)x_a+L_a(r)x_{a+1},
 \qquad
 c^{(a)}=c-(p_{a+1,1},\ldots,p_{a+1,n}).
\]
Then
\begin{equation}\label{eq:master-Da}
 9\Da(P)=9\Ra(P)+3\mathcal E(P),\qquad
 \mathcal E(P)=\sum_{a=1}^{m-2}\left\{
 z_a^{\mathsf T}K(c)z_a
 -s_a^{-2}z_a^{\mathsf T}K(c^{(a)})z_a\right\}.
\end{equation}
If $m=2$, the sum defining $\mathcal E(P)$ is empty; if $n=2$, every summand
vanishes.  Hence $\mathcal E(P)=0$ in either case.
\end{proposition}

\begin{proof}
For $\alpha=\cc_a$ or $\ggap_a$, set $a^+=a+1$ or $a+2$, respectively;
for $\beta=\cc_b$ or $\ggap_b$, define $b^+$ analogously.  The corresponding
collapsed $2\times2$ table is
\[
 P^{\alpha\beta}=
 \begin{pmatrix}
  P(X\leq a,Y\leq b)&P(X\leq a,Y\geq b^+)\\
  P(X\geq a^+,Y\leq b)&P(X\geq a^+,Y\geq b^+)
 \end{pmatrix}.
\]

\noindent\emph{The $\ts$ expansion.}
If $Z=(X,Y)$ and $Z'=(X',Y')$ are independent draws from $P$, then
\[
 \E\{\rho_\alpha(X,X')\rho_\beta(Y,Y')\}=2\det P^{\alpha\beta}.
\]
Apply \eqref{eq:signed-cut-gap} in both coordinates of \eqref{eq:tau-def}.
The pairs $(Z_1,Z_4)$ and $(Z_2,Z_3)$ are independent.  Grouping cut--cut
and gap--gap as same-type, and cut--gap and gap--cut as mixed-type, each
product of the two pair factors contributes $\{2\det P^{\alpha\beta}\}^2$.
Hence
\begin{equation}\label{eq:tau-collapse}
 \frac{\ts(P)}4=
 \sum_{\alpha,\beta\,\mathrm{same}}
 (\det P^{\alpha\beta})^2
 -\sum_{\alpha,\beta\,\mathrm{mixed}}
 (\det P^{\alpha\beta})^2.
\end{equation}

The four blocks of determinants in \eqref{eq:tau-collapse}, with their index
ranges displayed in the row and column headings, are
\begin{equation}\label{eq:actual-blocks}
\begin{array}{c|cc}
 &\cc_b,\ 1\leq b\leq n-1&\ggap_b,\ 1\leq b\leq n-2\\ \hline
 \cc_a,\ 1\leq a\leq m-1
   &x_{ab}&(\Xc G(c)^{\mathsf T})_{ab}\\
 \ggap_a,\ 1\leq a\leq m-2
   &(G(r)\Xc)_{ab}&s_a^{-1}\{G(c^{(a)})z_a\}_b
\end{array}
\end{equation}
The first is the determinant of the ordinary $2\times2$ cut table.  For
fixed $a$, collapsing the rows of $P$ into $\{i\leq a\}$ and $\{i>a\}$
gives a $2\times n$ table with margins $c$ and cut vector $x_a$;
Lemma~\ref{lem:cut-gap-transform} gives cut--gap.  The same construction for
$P^{\mathsf T}$ gives gap--cut.  Replacing the row blocks by $\{i\leq a\}$ and
$\{i\geq a+2\}$ gives a $2\times n$ table of mass $s_a$, with margins
$c^{(a)}$.  Its $b$th cut determinant is
\[
 U_a(r)P(X\leq a,Y\leq b)-L_a(r)P(X\geq a+2,Y\leq b)
 =U_a(r)x_{ab}+L_a(r)x_{a+1,b}=(z_a)_b,
\]
so its cut vector is $z_a$ and Lemma~\ref{lem:cut-gap-transform} gives gap--gap.

\medskip
\noindent\emph{Comparison.}
The first three determinant blocks in \eqref{eq:actual-blocks} agree with the
coefficient blocks used to derive \eqref{eq:master-Ra}; only the gap--gap
blocks differ.  Thus
\begin{equation}\label{eq:tau-four-blocks}
\begin{aligned}
 \frac{\ts(P)}4
 &=\|\Xc\|_F^2-\|\Xc G(c)^{\mathsf T}\|_F^2-\|G(r)\Xc\|_F^2
 +\sum_{a=1}^{m-2}s_a^{-2}\|G(c^{(a)})z_a\|_2^2.
\end{aligned}
\end{equation}
The rows of $G(r)\Xc$ are $z_a^{\mathsf T}$; in particular,
\[
 \|G(r)\Xc G(c)^{\mathsf T}\|_F^2
 =\sum_{a=1}^{m-2}\|G(c)z_a\|_2^2.
\]
Since $c^{(a)}$ has mass $s_a$,
$K(c^{(a)})=s_a^2I-G(c^{(a)})^{\mathsf T}G(c^{(a)})$.
Comparing
\eqref{eq:tau-four-blocks} with the norm form of \eqref{eq:master-Ra} and
using the definition of $K$ gives directly
\begin{equation}\label{eq:tau-R-correction}
\begin{aligned}
 \frac{\ts(P)}4-9\Ra(P)
 &=\sum_{a=1}^{m-2}\left\{
 s_a^{-2}\|G(c^{(a)})z_a\|_2^2-\|G(c)z_a\|_2^2\right\}\\
 &=\sum_{a=1}^{m-2}\left\{
 z_a^{\mathsf T}K(c)z_a
 -s_a^{-2}z_a^{\mathsf T}K(c^{(a)})z_a\right\}
 =\mathcal E(P).
\end{aligned}
\end{equation}
By \eqref{eq:tau-decomposition},
$\ts/4-9\Ra=3(\Da-\Ra)$.  Thus \eqref{eq:tau-R-correction} gives
\eqref{eq:master-Da}.  If $m=2$, the sum defining $\mathcal E$ is empty.  If
$n=2$, then $K(c)=I_1$ and
$K(c^{(a)})=s_a^2I_1$, so again $\mathcal E=0$.
\end{proof}

The remainder $\mathcal E$ need not be nonnegative: for
$P=\tfrac13\Bigl(\begin{smallmatrix}0&0&1\\1&0&0\\0&1&0\end{smallmatrix}\Bigr)$,
direct substitution gives $\mathcal E(P)=-1/729$.  Thus nonnegativity cannot
follow by proving the remainder nonnegative; its terms must instead be
reorganised into a separate certificate.

\begin{lemma}[opposite-shift pairing identity]
\label{lem:opposite-shift-pairing}
Let $n\geq2$, let $q_1,q_2\in\R_+^n$ be probability vectors, and let
$G_0\in\R^{(n-2)\times(n-1)}$.  Put
\[
 \Psi(v)=\|v\|_2^2-\|G_0v\|_2^2.
\]
Suppose that $v\in\R^{n-1}$, $\xi\in\R^{n-2}$ and
$\lambda_1,\lambda_2>0$ satisfy
\[
 G(q_1)v=G_0v+\frac{\xi}{\lambda_1},
 \qquad
 G(q_2)v=G_0v-\frac{\xi}{\lambda_2}.
\]
Then
\[
 \frac{\lambda_1\lambda_2}{\lambda_1+\lambda_2}\,
 v^{\mathsf T}\{\lambda_2K(q_1)+\lambda_1K(q_2)\}v
 =\lambda_1\lambda_2\Psi(v)
 +2(\lambda_1-\lambda_2)\langle G_0v,\xi\rangle
 -\left(\frac{\lambda_1}{\lambda_2}
 +\frac{\lambda_2}{\lambda_1}-1\right)\|\xi\|_2^2.
\]
In particular, the left-hand side is nonnegative.
\end{lemma}

\begin{proof}
Because $q_1$ and $q_2$ have mass one,
$K(q_i)=I-G(q_i)^{\mathsf T}G(q_i)$.  The two quadratic forms are
\[
\begin{aligned}
 v^{\mathsf T}K(q_1)v
 &=\Psi(v)-\frac{2}{\lambda_1}\langle G_0v,\xi\rangle
   -\frac{1}{\lambda_1^2}\|\xi\|_2^2,\\
 v^{\mathsf T}K(q_2)v
 &=\Psi(v)+\frac{2}{\lambda_2}\langle G_0v,\xi\rangle
   -\frac{1}{\lambda_2^2}\|\xi\|_2^2.
\end{aligned}
\]
Multiply the $q_1$ and $q_2$ identities by
$\lambda_1\lambda_2^2/(\lambda_1+\lambda_2)$ and
$\lambda_1^2\lambda_2/(\lambda_1+\lambda_2)$, respectively, and add.  The
stated coefficient of $\|\xi\|_2^2$ follows from
$(\lambda_1^3+\lambda_2^3)/\{\lambda_1\lambda_2(\lambda_1+\lambda_2)\}
=\lambda_1/\lambda_2+\lambda_2/\lambda_1-1$.  Finally,
$K(q_1),K(q_2)\succeq0$ by Lemma~\ref{lem:one-dimensional}.
\end{proof}

\begin{lemma}[row-block shift identities]
\label{lem:row-block-shifts}
Let $P$ be a finite probability table with $m,n\geq3$, row and column margins
$r,c$, and residual rows $x_a\in\R^{n-1}$, where $x_0=x_m=0$.  Put
$G_0=G(c)$ and define
\[
 [W(v)u]_b=v_{b+1}u_b-v_bu_{b+1}\quad(1\leq b\leq n-2),
 \qquad \xi_a=W(x_a)x_{a+1}.
\]
Here $0\leq a\leq m-1$ for $\xi_a$.  Also write
$\nu^{(a)}=p_{a,\boldsymbol\cdot}+p_{a+1,\boldsymbol\cdot}$ and
$t_a=r_a+r_{a+1}$ for $1\leq a\leq m-1$; for $1\leq a\leq m-2$, write
$z_a=U_a(r)x_a+L_a(r)x_{a+1}$, $s_a=1-r_{a+1}$, and
$c^{(a)}=c-p_{a+1,\boldsymbol\cdot}$.  Then
\begin{align}
 G(p_{i,\boldsymbol\cdot})
 &=r_iG_0-W(x_i-x_{i-1}), &&1\leq i\leq m,
 \label{eq:single-row-shift}\\
 G(\nu^{(a)})x_a
 &=t_aG_0x_a+\xi_{a-1}+\xi_a, &&1\leq a\leq m-1,
 \label{eq:adjacent-pair-shift}\\
 G(c^{(a)})z_a
 &=s_a(G_0z_a-\xi_a), &&1\leq a\leq m-2.
 \label{eq:omitted-row-shift}
\end{align}
\end{lemma}

\begin{proof}
Put $d_i=x_i-x_{i-1}$.  By the definition of the residuals,
$(d_i)_b=L_b(p_{i,\boldsymbol\cdot})-r_iL_b(c)$.  Thus the two nonzero
entries in row $b$ of $r_iG_0-W(d_i)$ are
\[
 r_iU_b(c)-(d_i)_{b+1}=U_b(p_{i,\boldsymbol\cdot}),
 \qquad
 r_iL_b(c)+(d_i)_b=L_b(p_{i,\boldsymbol\cdot}),
\]
proving \eqref{eq:single-row-shift}.  Summing it for rows $a,a+1$ and applying
the result to $x_a$ gives, by $W(u)v=-W(v)u$,
\[
 G(\nu^{(a)})x_a
 =t_aG_0x_a-W(x_{a+1}-x_{a-1})x_a
 =t_aG_0x_a+\xi_{a-1}+\xi_a,
\]
which is \eqref{eq:adjacent-pair-shift}.  Finally,
$G(c^{(a)})=s_aG_0+W(x_{a+1}-x_a)$ by
\eqref{eq:single-row-shift}, while $W(v)v=0$ gives
\[
 W(x_{a+1}-x_a)z_a
 =-\{U_a(r)+L_a(r)\}\xi_a=-s_a\xi_a,
\]
and hence \eqref{eq:omitted-row-shift}.
\end{proof}

\begin{proof}[Proof of Proposition~\ref{prop:finite-D}]
The assertion is trivial if only one row or column is attained.  Hence assume
$m,n\geq2$; by hypothesis all margins are positive.

\medskip
\noindent\emph{Boundary reduction.}
At $(i,j)$ the four boundary residuals are
$x_{ij},x_{i-1,j},x_{i,j-1},x_{i-1,j-1}$.  Averaging their squares with
respect to the joint mass $p_{ij}$ and shifting indices gives
\begin{equation}\label{eq:finite-D-boundary-sum}
 \sum_{\epsilon,\eta\in\{<,\leq\}}D_{\epsilon,\eta}(P)
 =\sum_{a=1}^{m-1}\sum_{b=1}^{n-1}
 (p_{ab}+p_{a+1,b}+p_{a,b+1}+p_{a+1,b+1})x_{ab}^2.
\end{equation}

All quantities in \eqref{eq:finite-D-boundary} are unchanged by transposing
$P$.  Hence, if $\min(m,n)=2$, we may assume without loss of generality that
$n=2$.
Proposition~\ref{prop:master} gives $\Da(P)=\Ra(P)$, while the preceding
display reduces to
$\sum_{a=1}^{m-1}(r_a+r_{a+1})x_{a1}^2$,
which is also $\sum_{\epsilon,\eta}R_{\epsilon,\eta}(P)$ by
\eqref{eq:finite-R-boundary-sum}.  The assertion therefore follows from
Proposition~\ref{prop:finite-R}.  Henceforth assume $m,n\geq3$.

\medskip
\noindent\emph{Adjacent-row comparison.}
For $1\leq a\leq m-1$, retain the notation
$\nu^{(a)}=p_{a,\boldsymbol\cdot}+p_{a+1,\boldsymbol\cdot}$ and
$t_a=r_a+r_{a+1}$ from Lemma~\ref{lem:row-block-shifts}.
Thus $t_a$ is the mass of the adjacent row pair, whereas
$s_a=1-r_{a+1}$ in Proposition~\ref{prop:master} is the mass outside the
row $a+1$.
Define the adjacent-row energy
\[
 J(P)=\sum_{a=1}^{m-1}
 \frac{x_a^{\mathsf T}K(\nu^{(a)})x_a}{t_a}.
\]
By Lemma~\ref{lem:one-dimensional} and \eqref{eq:finite-D-boundary-sum},
\[
 J(P)\geq\sum_{a=1}^{m-1}\sum_{b=1}^{n-1}
 (\nu_b^{(a)}+\nu_{b+1}^{(a)})x_{ab}^2
 =\sum_{\epsilon,\eta\in\{<,\leq\}}D_{\epsilon,\eta}(P).
\]
It remains to prove $9\Da(P)\geq J(P)$.  We shall express both quantities
relative to the same quadratic baseline.  Their difference will split into
two manifestly nonnegative sums and local terms that will be certified by a
three-block argument.

\medskip
\noindent\emph{Adjacent-determinant reduction.}
Retain $G_0$, $W$ and $\xi_a$ from
Lemma~\ref{lem:row-block-shifts}.  Since $c$ has mass one, define
\[
 \Psi(v)=v^{\mathsf T}K(c)v=\|v\|_2^2-\|G_0v\|_2^2.
\]
The common baseline just described is
$\sum_{a=1}^{m-1}t_a\Psi(x_a)$.
We shall repeatedly use the following shift identity.  If $q\in\R_+^n$ has
mass $t>0$ and $G(q)v=tG_0v+\eta$, then, since
$K(q)=t^2I-G(q)^{\mathsf T}G(q)$,
\begin{equation}\label{eq:mass-shift}
 v^{\mathsf T}K(q)v
 =t^2\|v\|_2^2-\|tG_0v+\eta\|_2^2
 =t^2\Psi(v)-2t\langle G_0v,\eta\rangle-\|\eta\|_2^2.
\end{equation}

Applying \eqref{eq:mass-shift} to \eqref{eq:adjacent-pair-shift}, dividing by
$t_a$, and
summing over $a$ gives
\[
 J(P)=\sum_{a=1}^{m-1}t_a\Psi(x_a)
 -2\sum_{a=1}^{m-1}\langle G_0x_a,\xi_{a-1}+\xi_a\rangle
 -\sum_{a=1}^{m-1}\frac{\|\xi_{a-1}+\xi_a\|_2^2}{t_a}.
\]

For the rest of the proof, abbreviate $L_a=L_a(r)$ and $U_a=U_a(r)$.
For $1\leq a\leq m-2$, recall from
Lemma~\ref{lem:row-block-shifts} that
$z_a=U_ax_a+L_ax_{a+1}$ and $s_a=L_a+U_a=1-r_{a+1}$.
Equation~\eqref{eq:omitted-row-shift} gives
\[
 G(c^{(a)})z_a
 =s_a(G_0z_a-\xi_a).
\]
Applying \eqref{eq:mass-shift} with $q=c^{(a)}$, $t=s_a$, $v=z_a$, and
$\eta=-s_a\xi_a$ first yields the auditable intermediate identity
\[
 s_a^{-2}z_a^{\mathsf T}K(c^{(a)})z_a
 =\Psi(z_a)+2\langle G_0z_a,\xi_a\rangle-\|\xi_a\|_2^2.
\]
Using $z_a^{\mathsf T}K(c)z_a=\Psi(z_a)$ in
Proposition~\ref{prop:master} now gives
\begin{equation}\label{eq:E-adjacent-determinant}
 \mathcal E(P)=\sum_{a=1}^{m-2}
 \left\{\|\xi_a\|_2^2-2\langle G_0z_a,\xi_a\rangle\right\}.
\end{equation}

\medskip
\noindent\emph{Organizing decomposition.}
We now name the local contributions that will remain in $9\Da(P)-J(P)$.
For $1\leq a\leq m-2$, put
\[
\begin{aligned}
 \mathcal B_a
 &={L_aU_a}\Psi(x_a-x_{a+1})
 +r_{a+1}\{U_a\Psi(x_a)+L_a\Psi(x_{a+1})\},
 \\
 \mathcal C_a&=2\langle G_0(x_a+x_{a+1}-3z_a),\xi_a\rangle.
\end{aligned}
\]
Thus $\mathcal B_a$ will be the row-energy excess, $\mathcal C_a$ the combined
mixed term.  The three-block argument below will prove
$\mathcal B_a+\mathcal C_a\geq0$.

Lemma~\ref{lem:one-dimensional} gives $K(c)\succeq0$.  Put
$Y=\Xc K(c)^{1/2}$, and let $y^{(j)}$ denote its $j$th column.  The squared
norms of the $a$th row of $Y$ and of the difference between its $a$th and
$(a+1)$st rows are, respectively, $\Psi(x_a)$ and
$\Psi(x_a-x_{a+1})$.  Equations~\eqref{eq:master-Ra} and
\eqref{eq:weighted-squares}, with $w=r$ and $s=1$, now give the exact
row-energy decomposition
\begin{equation}\label{eq:Ra-row-energy}
 9\Ra(P)=\operatorname{tr}\{K(r)\Xc K(c)\Xc^{\mathsf T}\}
 =\sum_{j=1}^{n-1}(y^{(j)})^{\mathsf T}K(r)y^{(j)}
 =\sum_{a=1}^{m-1}t_a\Psi(x_a)
  +\sum_{a=1}^{m-2}\mathcal B_a.
\end{equation}

For a coefficient ledger, write
$\mathcal A_0=\sum_{a=1}^{m-1}t_a\Psi(x_a)$.  The three sources in
$9\Da(P)-J(P)$ are
\[
\begin{aligned}
 9\Ra(P)
 &=\mathcal A_0+\sum_{a=1}^{m-2}\mathcal B_a,\\
 3\mathcal E(P)
 &=-6\sum_{a=1}^{m-2}\langle G_0z_a,\xi_a\rangle
   +3\sum_{a=1}^{m-2}\|\xi_a\|_2^2,\\
 -J(P)
 &=-\mathcal A_0
   +2\sum_{a=1}^{m-1}\langle G_0x_a,\xi_{a-1}+\xi_a\rangle
   +\sum_{a=1}^{m-1}\frac{\|\xi_{a-1}+\xi_a\|_2^2}{t_a}.
\end{aligned}
\]
Because $\xi_0=\xi_{m-1}=0$, an index shift in the cross terms gives
\[
 2\sum_{a=1}^{m-1}\langle G_0x_a,\xi_{a-1}+\xi_a\rangle
 -6\sum_{a=1}^{m-2}\langle G_0z_a,\xi_a\rangle
 =2\sum_{a=1}^{m-2}
 \langle G_0(x_a+x_{a+1}-3z_a),\xi_a\rangle
 =\sum_{a=1}^{m-2}\mathcal C_a.
\]
Consequently, \eqref{eq:master-Da}, \eqref{eq:E-adjacent-determinant},
\eqref{eq:Ra-row-energy}, and the earlier expansion of $J(P)$ give
\begin{equation}\label{eq:pre-certificate}
\begin{aligned}
 9\Da(P)&=\sum_{a=1}^{m-1}t_a\Psi(x_a)
 -2\sum_{a=1}^{m-1}\langle G_0x_a,\xi_{a-1}+\xi_a\rangle
 +\sum_{a=1}^{m-2}(\mathcal B_a+\mathcal C_a)
 +3\sum_{a=1}^{m-2}\|\xi_a\|_2^2\\
 &=J(P)+\sum_{a=1}^{m-2}(\mathcal B_a+\mathcal C_a)
 +3\sum_{a=1}^{m-2}\|\xi_a\|_2^2
 +\sum_{a=1}^{m-1}\frac{\|\xi_{a-1}+\xi_a\|_2^2}{t_a}.
\end{aligned}
\end{equation}

The last two sums are manifestly nonnegative.  It remains to prove
$\mathcal B_a+\mathcal C_a\geq0$ for every $a$.

\medskip
\noindent\emph{Three-block setup.}
For $1\leq a\leq m-2$, abbreviate
\[
 \ell=L_a,\qquad m_0=r_{a+1},\qquad u=U_a,
 \qquad x=x_a,\quad y=x_{a+1}.
\]
Since all row levels are attained, $\ell,m_0,u>0$.
The conditional laws of $Y$ given the lower, middle and upper row blocks are
\[
 c_L=\frac1\ell\sum_{i\leq a}p_{i,\boldsymbol\cdot},\qquad
 c_M=\frac1{m_0}p_{a+1,\boldsymbol\cdot},\qquad
 c_U=\frac1u\sum_{i\geq a+2}p_{i,\boldsymbol\cdot}.
\]
Thus $c=\ell c_L+m_0c_M+u c_U$.  Summing
\eqref{eq:single-row-shift} over the three row blocks, using $x_0=x_m=0$, and
dividing by their masses gives
\[
 G(c_L)=G_0-\frac1\ell W(x),\qquad
 G(c_M)=G_0-\frac1{m_0}W(y-x),\qquad
 G(c_U)=G_0+\frac1uW(y).
\]
The matrices $K(c_L),K(c_M),K(c_U)$ are positive semidefinite by
Lemma~\ref{lem:one-dimensional}.  Antisymmetry of $W$ and $W(v)v=0$ now give
the following array of products $G(q)v$:
\begin{equation}\label{eq:three-block-shifts}
\begin{array}{c|ccc}
 G(q)v&v=x&v=y&v=x-y\\ \hline
 q=c_L&G_0x&G_0y-\xi_a/\ell&G_0(x-y)+\xi_a/\ell\\
 q=c_M&G_0x+\xi_a/m_0&G_0y+\xi_a/m_0&G_0(x-y)\\
 q=c_U&G_0x-\xi_a/u&G_0y&G_0(x-y)-\xi_a/u.
\end{array}
\end{equation}

\medskip
\noindent\emph{Pairing the three blocks.}
Apply Lemma~\ref{lem:opposite-shift-pairing} to the two nonzero shifts in
each column of \eqref{eq:three-block-shifts}, always with $\xi=\xi_a$.
For $v=x$, take
$(q_1,q_2;\lambda_1,\lambda_2)=(c_M,c_U;m_0,u)$; for $v=y$ and $v=x-y$,
take $(c_M,c_L;m_0,\ell)$ and $(c_L,c_U;\ell,u)$, respectively.
The three nonnegative left sides sum to
\[
\begin{aligned}
 \mathcal N_a={}&
 \frac{m_0u}{m_0+u}\,
 x^{\mathsf T}\{uK(c_M)+m_0K(c_U)\}x
 +\frac{\ell m_0}{\ell+m_0}\,
 y^{\mathsf T}\{\ell K(c_M)+m_0K(c_L)\}y\\
 &+\frac{\ell u}{\ell+u}\,
 (x-y)^{\mathsf T}\{uK(c_L)+\ell K(c_U)\}(x-y).
\end{aligned}
\]
Thus $\mathcal N_a\geq0$.  The $\Psi$-terms on the three right sides satisfy
\[
 m_0u\Psi(x)+m_0\ell\Psi(y)+\ell u\Psi(x-y)=\mathcal B_a.
\]
Their cross terms satisfy
\[
\begin{aligned}
 &2\left\langle
 (m_0-u)G_0x+(m_0-\ell)G_0y+(\ell-u)G_0(x-y),\xi_a
 \right\rangle\\
 &\qquad=2\langle G_0(x+y-3z_a),\xi_a\rangle=\mathcal C_a,
\end{aligned}
\]
where the first equality uses $\ell+m_0+u=1$ and
$z_a=ux+\ell y$.
Finally, since $\ell+m_0+u=1$, the coefficients multiplying
$-\|\xi_a\|_2^2$ add to
\[
 \left(\frac{m_0}{u}+\frac{u}{m_0}-1\right)
 +\left(\frac{m_0}{\ell}+\frac{\ell}{m_0}-1\right)
 +\left(\frac{\ell}{u}+\frac{u}{\ell}-1\right)
 =\frac1\ell+\frac1{m_0}+\frac1u-6.
\]
Therefore
\begin{equation}\label{eq:local-three-block-certificate}
 \mathcal B_a+\mathcal C_a
 =\mathcal N_a+
 \left(\frac1\ell+\frac1{m_0}+\frac1u-6\right)\|\xi_a\|_2^2.
\end{equation}

Since $\ell+m_0+u=1$, AM--HM gives
$\ell^{-1}+m_0^{-1}+u^{-1}\geq9$, so the coefficient in
\eqref{eq:local-three-block-certificate} is at least $3$.  Hence
\eqref{eq:local-three-block-certificate} and $\mathcal N_a\geq0$ imply
$\mathcal B_a+\mathcal C_a\geq0$.  Equation~\eqref{eq:pre-certificate}
therefore gives $9\Da(P)\geq J(P)$.  Together with the adjacent-row comparison
above, this proves
\eqref{eq:finite-D-boundary}.  The displayed consequence follows from
\eqref{eq:tau-decomposition} and Proposition~\ref{prop:finite-R}.
\end{proof}

\section{Passage to arbitrary laws, zero sets, and sharpness}
\label{sec:arbitrary-laws}

Weak-order quantisation transfers the finite inequalities without losing
information about ties and preserves the boundary functionals needed for the
equality case.  The final subsection applies these limits and proves
sharpness.

\subsection{Weak-order quantisation}

For integers $k\geq1$, define the finite-range nondecreasing quantiser
\[
 q_k(x)=2^{-k}
 \left\lfloor2^k\bigl((-k)\vee(x\wedge k)\bigr)\right\rfloor,
 \qquad X^{(k)}=q_k(X),\quad Y^{(k)}=q_k(Y).
\]
Thus $X^{(k)}$ and $Y^{(k)}$ are obtained by clipping $X$ and $Y$ to
$[-k,k]$ and rounding down to the dyadic grid $2^{-k}\mathbb Z$; in
particular, they have finite range and converge pointwise to $X$ and $Y$.

\begin{lemma}[weak-order quantisation]\label{lem:quantization}
For every bivariate law,
\[
 \begin{gathered}
 \ts(X^{(k)},Y^{(k)})\longrightarrow\ts(X,Y),\qquad
 \Ra(X^{(k)},Y^{(k)})\longrightarrow\Ra(X,Y),\\
 \Da(X^{(k)},Y^{(k)})\longrightarrow\Da(X,Y),\\
 D_{\epsilon,\eta}(X^{(k)},Y^{(k)})\longrightarrow D_{\epsilon,\eta}(X,Y),\qquad
 R_{\epsilon,\eta}(X^{(k)},Y^{(k)})\longrightarrow R_{\epsilon,\eta}(X,Y),\\
 \epsilon,\eta\in\{<,\leq\}.
 \end{gathered}
\]
\end{lemma}

\begin{proof}
Fix one iid sequence $(X_i,Y_i)_{i\geq1}$ with the law of $(X,Y)$, and put
$X_i^{(k)}=q_k(X_i)$ and $Y_i^{(k)}=q_k(Y_i)$.  Thus all comparisons may be
made pathwise on a common probability space, while, for each fixed $k$,
$(X_i^{(k)},Y_i^{(k)})_{i\geq1}$ is iid with common law
$\mathcal L(X^{(k)},Y^{(k)})$.

\medskip
\noindent\emph{Weak-order stabilisation.}
Fix a finite realised coordinate list.  If it has more than one distinct
value, their minimum positive separation is positive; if not, preservation is
immediate.  Hence, for all sufficiently large $k$, depending on the realised
list, clipping is inactive and the mesh $2^{-k}$ is smaller than every positive
separation.  The quantiser then preserves every strict comparison, while ties
are preserved identically.
Applying this pathwise argument separately to the two coordinate lists shows
that, almost surely, the quantised lists
$(X_i^{(k)})_{1\leq i\leq p}$ and $(Y_i^{(k)})_{1\leq i\leq p}$ have exactly
the same complete weak orders as their unquantised counterparts for all
sufficiently large $k$.  Consequently, any bounded kernel determined by these
weak orders is eventually unchanged almost surely, and bounded convergence
applies.

By Lemma~\ref{lem:cut-gap} and \eqref{eq:Da-rectangular}, the bounded kernels
defining $\ts$, $\Ra$ and $\Da$ depend only on the complete weak orders in
the two coordinates.  The stabilisation argument above and bounded
convergence therefore prove all three limits.

It remains to treat the boundary functionals.  Fix
$\epsilon,\eta\in\{<,\leq\}$.  For $(x,y)\in\R^2$, put
\[
\begin{aligned}
 C_1(x,y)
 &=\iota_\epsilon(X_2,x)
   \{\iota_\eta(Y_2,y)-\iota_\eta(Y_4,y)\},\\
 C_2(x,y)
 &=\iota_\epsilon(X_3,x)
   \{\iota_\eta(Y_3,y)-\iota_\eta(Y_5,y)\}.
\end{aligned}
\]
For fixed $(x,y)$, the variables $C_1(x,y)$ and $C_2(x,y)$ are independent
and both have mean $\Delta_{\epsilon,\eta}(x,y)$.  Since $(X_1,Y_1)$ and
$(X_1,Y_6)$ are independent of the copies used in $C_1$ and $C_2$, while the
components of $(X_1,Y_6)$ are independent, it follows that
\[
 D_{\epsilon,\eta}=\E\{C_1(X_1,Y_1)C_2(X_1,Y_1)\},\qquad
 R_{\epsilon,\eta}=\E\{C_1(X_1,Y_6)C_2(X_1,Y_6)\}.
\]
The first product is a bounded five-copy kernel and the second a bounded
six-copy kernel.  Both depend only on complete weak orders, so the stabilisation
argument above proves all eight boundary limits.
\end{proof}

\begin{remark}[Why convergence alone is insufficient]
The limits above do not follow merely from $X^{(k)}\to X$ and $Y^{(k)}\to Y$
almost surely, because the kernels are discontinuous at ties.  For example,
let $Z$ be uniform on $\{1,2\}$, let $U$ be uniform on $[-1,1]$ independently
of $Z$, and set $Z_n=Z+\varepsilon_nU$ for $\varepsilon_n\downarrow0$.
Then $Z_n\to Z$ almost surely, whereas a direct calculation gives
$\ts(Z,Z)=1/4$ and $\ts(Z_n,Z_n)=2/3$ for every $n$.
The proof works because, on each finite sample, quantisation eventually
preserves the weak orders exactly, making the relevant kernels identical
rather than merely close.
\end{remark}

\subsection{Completion and sharpness}

\begin{proof}[Proof of Theorems~\ref{thm:main} and~\ref{thm:four-boundary}]
\noindent\emph{Boundary inequalities and nonnegativity.}
The atomless identifications and the decomposition
\eqref{eq:tau-decomposition} were established above.
For each $k$, discard unattained grid levels, order the attained row and
column values, and relabel them increasingly and consecutively as
$1,\ldots,m_k$ and $1,\ldots,n_k$.  All the kernel and boundary functionals
involved are unchanged by this strictly increasing relabelling, so the law of
$(X^{(k)},Y^{(k)})$ becomes a finite probability table to which
Propositions~\ref{prop:finite-R} and~\ref{prop:finite-D} apply.  Letting
$k\to\infty$ and using Lemma~\ref{lem:quantization} proves
\eqref{eq:four-boundary-main}.  Since every boundary functional is
nonnegative, it also proves $\Da(X,Y)\geq0$ and $\Ra(X,Y)\geq0$.  The
strengthened lower bound for $\ts$ follows from
\eqref{eq:tau-decomposition}.

\medskip
\noindent\emph{The zero sets of $\Ra$ and $\ts$.}
Under independence, the defining expectations factor into products of the
centred marginal expectations, so $\ts=\Da=\Ra=0$.  Conversely, if $\Ra=0$, then
\eqref{eq:four-boundary-main} gives $R_{\leq,\leq}=0$, hence independence
because $R_{\leq,\leq}=R$ has zero set independence
\citep[Thm.~2]{WDM}.  Thus $\Ra=0$ if and only if $X\indep Y$.  Since
$\Da,\Ra\geq0$, \eqref{eq:tau-decomposition} gives $\ts\geq0$, with equality
if and only if $X\indep Y$.

\medskip
\noindent\emph{The $\Da$ zero set for absolutely continuous laws.}
If the joint law of $(X,Y)$ is absolutely continuous, then its margins are
atomless and hence $\Da=D$.  Assuming only this joint absolute continuity,
with no positivity or continuity requirement on the density,
\citet[Prop.~3 and proof, pp.~58--59]{Yanagimoto} proves that his
$M_1(F)=D$ vanishes if and only if $F(x,y)=F_X(x)F_Y(y)$ for every
$(x,y)\in\R^2$, which is independence.  Together with the converse established
above, this proves the stated zero-set characterisation in the absolutely
continuous case.

\medskip
\noindent\emph{The $\Da$ zero set for purely atomic laws.}
It remains to prove the stated refinement for purely atomic laws.  Suppose
that the law of $(X,Y)$ is purely atomic and $\Da=0$.  Since each
$D_{\epsilon,\eta}$ is nonnegative, \eqref{eq:four-boundary-main} gives
$D_{\epsilon,\eta}=0$ for every $\epsilon,\eta\in\{<,\leq\}$.
Let
\[
 A=\{x:\Pr(X=x)>0\},\qquad B=\{y:\Pr(Y=y)>0\},
\]
and write $r_x=\Pr(X=x)$, $c_y=\Pr(Y=y)$ and
$p_{xy}=\Pr(X=x,Y=y)$, and put
$S=\{(x,y)\in A\times B:p_{xy}>0\}$.  The sets $A$ and $B$ are at most
countable, and
\[
 D_{\epsilon,\eta}
 =\sum_{(x,y)\in S}p_{xy}\Delta_{\epsilon,\eta}(x,y)^2.
\]
Hence all four discrepancies vanish at every $(x,y)\in S$.
Inclusion--exclusion then gives
\[ p_{xy}-r_xc_y=\Delta_{\leq,\leq}(x,y)-\Delta_{<,\leq}(x,y)-\Delta_{\leq,<}(x,y)+\Delta_{<,<}(x,y)=0. \]
Consequently,
\[
 1=\sum_{(x,y)\in S}p_{xy}
  =\sum_{(x,y)\in S}r_xc_y
  \leq\sum_{(x,y)\in A\times B}r_xc_y=1.
\]
Since $r_xc_y>0$ throughout $A\times B$, equality forces $S=A\times B$.
Thus $p_{xy}=r_xc_y$ throughout $A\times B$, proving independence.  The
converse was established above.  This completes the qualitative assertions
of Theorem~\ref{thm:main}.

\medskip
\noindent\emph{Sharp constants.}
It remains to verify sharpness.  For a $2\times2$ table
$P=(p_{ij})_{i,j=1}^2$, put
$\theta=p_{11}p_{22}-p_{12}p_{21}$.  The finite formulas give
\[
 \ts=4\theta^2,\qquad \Ra=\Da=\frac{\theta^2}{9},
\]
and
\[
 D_{\leq,\leq}=p_{11}\theta^2,\quad
 D_{\leq,<}=p_{12}\theta^2,\quad
 D_{<,\leq}=p_{21}\theta^2,\quad
 D_{<,<}=p_{22}\theta^2.
\]
Likewise,
\[
 R_{\leq,\leq}=r_1c_1\theta^2,\quad
 R_{\leq,<}=r_1c_2\theta^2,\quad
 R_{<,\leq}=r_2c_1\theta^2,\quad
 R_{<,<}=r_2c_2\theta^2.
\]
Both boundary sums therefore equal $\theta^2=9\Ra=9\Da$.  Thus equality
holds in both inequalities in \eqref{eq:four-boundary-main} for every
dependent $2\times2$ table, proving that both constants $9$ are best possible.

For the coefficient $24$, \citet[Remark~1, p.~59]{Yanagimoto}
constructs a continuous bivariate cdf $F_0$ with both margins uniform on
$(0,1)$ and with Hoeffding functional $D(F_0)=0$, but
\[
 F_0\!\left(\frac12,\frac14\right)\ne\frac18
 =F_{0,X}\!\left(\frac12\right)F_{0,Y}\!\left(\frac14\right),
\]
so the law is dependent.  Since its margins are atomless, $\Da=D=0$; by the
absolutely continuous zero-set result just proved, its joint law cannot be
absolutely continuous.  Its uniform margins also rule out a purely atomic
joint law, so the example lies outside the two subclasses for which $\Da$
characterises independence.  Dependence and the equivalence
$\Ra=0\Longleftrightarrow X\indep Y$ give $\Ra>0$.  The identity
\eqref{eq:tau-decomposition} therefore becomes
$\ts=24\Ra$, proving
sharpness of $24$, even with the four-boundary term retained.
Moreover, apply the quantisers above to this Yanagimoto law.  By
Lemma~\ref{lem:quantization},
\[
 \Da(X^{(k)},Y^{(k)})\longrightarrow0,
 \qquad
 \Ra(X^{(k)},Y^{(k)})\longrightarrow\Ra(X,Y)>0.
\]
Since their $\Ra$-values are eventually positive, the quantised laws are
finite and dependent for all sufficiently large $k$, and
\[
 \frac{\ts(X^{(k)},Y^{(k)})}{\Ra(X^{(k)},Y^{(k)})}
 =24+12\frac{\Da(X^{(k)},Y^{(k)})}{\Ra(X^{(k)},Y^{(k)})}
 \longrightarrow24.
\]
Thus $24$ is best possible even over finite probability tables.  On the other
hand, the purely atomic zero-set result and \eqref{eq:tau-decomposition} show
that the inequality is strict for every dependent purely atomic law.
\end{proof}

\section{Further consequences and comparison}\label{sec:further-consequences}

We conclude with a monotone-association analogue and its quadratic bounds, the
effect of ties on normalisation, further structural and multivariate
consequences, and an outlook.

\subsection{A monotone-association analogue}
\label{subsec:monotone-association}

Write $s_Z(ij)=\sgn(Z_i-Z_j)$.  Kendall's tau and the sign version of
Spearman's rho are
\[
 \tau(X,Y)=\E\{s_X(12)s_Y(12)\},\qquad
 \rho_S(X,Y)=3\E\{s_X(12)s_Y(13)\}.
\]
Define
\[
 \upsilon(X,Y)=\frac34\E\bigl[
   \{s_X(12)-s_X(13)\}\{s_Y(12)-s_Y(13)\}\bigr].
\]
Exchangeability and expansion of the product give, for every bivariate law,
\begin{equation}\label{eq:linear-tau-decomposition}
 \boxed{\qquad \tau(X,Y)=\frac13\rho_S(X,Y)+\frac23\upsilon(X,Y).\qquad}
\end{equation}
This is a conditional-covariance analogue of
Lemma~\ref{lem:pair-projection}.
By replicate independence, conditioning $U=s_X(12)$ and $V=s_Y(12)$ on
$(X_1,Y_1)$ identifies the product-of-conditional-means term as
$\E\{s_X(12)s_Y(13)\}=\rho_S/3$ and the conditional-covariance term
$2\upsilon/3$.
All three functionals vanish under independence.  Writing
$p_z=\Pr(Z=z)$, with sums over the atoms of $Z$, direct conditioning gives
$\tau(Z,Z)=1-\sum_zp_z^2$, $\rho_S(Z,Z)=1-\sum_zp_z^3$, and
$\upsilon(Z,Z)=1-\tfrac12\sum_zp_z^2(3-p_z)$.  Their self-values therefore lie
in $[0,1]$.  Cauchy--Schwarz in the three product representations (conditioning
first on $(X_1,Y_1)$ for $\rho_S$) yields
$|T(X,Y)|\leq\{T(X,X)T(Y,Y)\}^{1/2}\leq1$ for
$T\in\{\tau,\rho_S,\upsilon\}$.
When the marginals are atomless and $Y$ is a strictly increasing or strictly
decreasing function of $X$, all three equal $1$ or $-1$, respectively.

For atomless marginals, the four strict/non-strict cdf discrepancies defined
above coincide; write
$\Delta(x,y)=\Delta_{\leq,\leq}(x,y)=F(x,y)-F_X(x)F_Y(y)$.
The two components and their sum have
the integral representations
\[
 \begin{gathered}
  \frac23\upsilon=4\int_{\R^2}\Delta\,dF,
  \qquad
  \frac13\rho_S=4\int_{\R^2}\Delta\,dF_X\,dF_Y,\\
  \tau=4\int_{\R^2}\Delta\,dF
       +4\int_{\R^2}\Delta\,dF_X\,dF_Y.
 \end{gathered}
\]
The displayed identity for $\tau$ is the linear monotone-association analogue
of \eqref{eq:tau-decomposition}, and both identities arise from the same
projection principle.  For $\tau$, however, the projected cross-components
are signed and may cancel.  Thus the nonnegative, cancellation-free form of
the $\ts$ decomposition is kernel-specific rather than a generic consequence
of orthogonality.  With ties, all three are unnormalised sign functionals:
$\tau$ is Kendall's $\tau_a$, while $\rho_S$ and $\upsilon$ are defined by the
displays above without marginal tie normalisation; see \citet{Neslehova} for
rank-correlation measures for non-continuous variables.

\subsection{Quadratic bounds from monotone association}

\begin{proposition}[Quadratic bounds from monotone association]
\label{prop:monotone-square-bounds}
\textup{(i)} If both margins are atomless, then
\[
 \Da\geq\frac{\upsilon^2}{36},
 \qquad
 \Ra\geq\frac{\rho_S^2}{144}.
\]
In fact, with $\Delta$ as above, the following exact remainder identity holds:
\begin{equation}\label{eq:atomless-tau-square-decomposition}
 \ts=\frac12\tau^2+\frac14(\rho_S-\tau)^2
 +12\int_{\R^2}\left(\Delta-\frac{\upsilon}{6}\right)^2\,dF
 +24\int_{\R^2}\left(\Delta-\frac{\rho_S}{12}\right)^2
       \,dF_X\,dF_Y.
\end{equation}
Consequently,
\begin{equation}\label{eq:atomless-tau-square}
 \boxed{\qquad
 \ts\geq\frac12\tau^2+\frac14(\rho_S-\tau)^2
 \geq\frac12\tau^2.
 \qquad}
\end{equation}

\textup{(ii)} For every bivariate law on $\R^2$,
\begin{equation}\label{eq:component-square-bounds}
 \boxed{\qquad
  \Da\geq\frac{\upsilon^2}{81},
  \qquad
  \Ra\geq\frac{\rho_S^2}{324}.
 \qquad}
\end{equation}
Consequently,
\begin{equation}\label{eq:general-tau-square}
 \boxed{\qquad
  \ts\geq
  \frac{2\tau^2+(\rho_S-\tau)^2}{9}
  \geq\frac29\tau^2.
 \qquad}
\end{equation}
The two constants in \eqref{eq:component-square-bounds}, and the first bound
in \eqref{eq:general-tau-square}, are sharp.  The stronger atomless bound
\eqref{eq:atomless-tau-square} does not extend to arbitrary laws.  If $c_*$ is
the best universal constant in $\ts\geq c_*\tau^2$, the present bounds show
$2/9\leq c_*\leq1/4$ but do not determine its exact value.
\end{proposition}

\begin{proof}
For part~(i), the preceding integral representations and the atomless
identities $\Da=D$, $\Ra=R$ give the exact constant-projection identities
\begin{equation}\label{eq:constant-projections}
 \Da=D=\frac{\upsilon^2}{36}
   +\int_{\R^2}\left(\Delta-\frac{\upsilon}{6}\right)^2\,dF,
 \qquad
 \Ra=R=\frac{\rho_S^2}{144}
   +\int_{\R^2}\left(\Delta-\frac{\rho_S}{12}\right)^2\,dF_X\,dF_Y.
\end{equation}
The identities in \eqref{eq:constant-projections} give the two component
bounds.  Substitution in
\eqref{eq:tau-decomposition}, followed by
$\upsilon=(3\tau-\rho_S)/2$ from
\eqref{eq:linear-tau-decomposition}, gives
\eqref{eq:atomless-tau-square-decomposition}.  Dropping its two nonnegative
remainders proves \eqref{eq:atomless-tau-square}.

For part~(ii), summing the four boundary discrepancies gives
$\operatorname{Cov}\{\sgn(x-X),\sgn(y-Y)\}$.  Its integrals under the joint
law and the product of the margins are $2\upsilon/3$ and $\rho_S/3$,
respectively.  Projection onto the one-dimensional subspace of common
constants gives, by Theorem~\ref{thm:four-boundary},
\[
 9\Da\geq\sum_{\epsilon,\eta}D_{\epsilon,\eta}
       \geq\frac{\upsilon^2}{9},
 \qquad
 9\Ra\geq\sum_{\epsilon,\eta}R_{\epsilon,\eta}
       \geq\frac{\rho_S^2}{36}.
\]
This proves \eqref{eq:component-square-bounds}; combining those bounds with
\eqref{eq:tau-decomposition} and
$\upsilon=(3\tau-\rho_S)/2$ proves \eqref{eq:general-tau-square}.  A three-point
perturbation of independence makes the two component bounds and the first
inequality in \eqref{eq:general-tau-square} asymptotically equalities and has
$\ts/\tau^2\to1/4$.  This establishes sharpness and shows why the stronger
atomless coefficient cannot hold for arbitrary laws.  The exact projection
identities and the calculation for this family are given in Section~3 of the
supplement.
\end{proof}

\subsection{Ranges and normalisation under ties}
\label{subsec:ranges-normalisation}

The constants used in the classical normalisation have different statuses for
the three functionals under ties.  The scales for $\ts$ and $\Da$ remain their
global sharp scales.  For $\Ra$, the classical atomless scale $1/90$ is not a
global upper bound; we give a universal upper bound below, but do not claim
that it is sharp.

\begin{proposition}[Ranges and scales under ties]\label{prop:sharp-scales}
For every bivariate law on $\R^2$,
\begin{equation}\label{eq:range-bounds}
 \begin{gathered}
  0\leq\ts(X,Y)\leq\frac23,
  \qquad 0\leq\Da(X,Y)\leq\frac1{30},\\
  0\leq\Ra(X,Y)
  \leq\{\Ra(X,X)\Ra(Y,Y)\}^{1/2}
  \leq\frac1{36}.
 \end{gathered}
\end{equation}
The first two upper bounds are attained when the marginals are atomless and one
variable is a strictly monotone function of the other; $1/36$ is not asserted
to be sharp.  Nevertheless, the
classical atomless value $1/90$ is not an upper bound for $\Ra$.  If $Z$ is uniform on
$\{1,\ldots,m\}$ for a positive integer $m$, then
\begin{equation}\label{eq:Ra-uniform-range-counterexample}
 \Ra(Z,Z)
 =\frac1{90}
  +\frac{m^4-15m^3+25m^2-18m+4}{270m^5}.
\end{equation}
 Consequently, $\Ra(Z,Z)>1/90$ for every integer $m\geq14$, and $14$ is
 minimal.  At $m=14$, $\Ra(Z,Z)=1/90+53/4033680>1/90$.
\end{proposition}

\begin{proof}
The lower bounds in \eqref{eq:range-bounds} follow from
Theorem~\ref{thm:main}.  For a real random variable $Z$, put
$A_Z=a_Z(1234)$ and let $E_1^Z,E_2^Z,E_3^Z$ denote the three separation events
in the case table of Subsection~\ref{subsec:projection-structure}.  Its first
row gives
\[
 A_Z^2=\1(E_1^Z)+\1(E_2^Z).
\]
By exchangeability, the three events have equal probabilities, and they are
mutually exclusive.  Hence, by Cauchy--Schwarz,
\[
 3\E A_Z^2=2\sum_{i=1}^3\Pr(E_i^Z)\leq2,
 \qquad
 \ts(X,Y)\leq\{\E A_X^2\,\E A_Y^2\}^{1/2}\leq\frac23.
\]

The product representation of $\Ra$ and Cauchy--Schwarz give the middle
inequality in \eqref{eq:range-bounds}.  For any real $Z$,
Theorem~\ref{thm:main} gives $\Da(Z,Z)\geq0$, so
\eqref{eq:tau-decomposition} and the just-proved bound for $\ts$ give
$24\Ra(Z,Z)\leq\ts(Z,Z)\leq2/3$, giving the last
inequality.

For the second upper bound, with $H_Z$ as in \eqref{eq:K-rectangular}, set
$d(Z)=\E\{\ell_Z(1234)^2\}=\E H_Z^2=\Da(Z,Z)$.  The product representation
\eqref{eq:Da-rectangular}, Cauchy--Schwarz and Theorem~\ref{thm:main} give
$0\leq\Da(X,Y)\leq\{d(X)d(Y)\}^{1/2}$.  For finitely supported $Z$, repeatedly
split each atom into two equal consecutive subatoms.  The lemma in Section~7
of the supplement gives $d(Z)\leq d(Z_N)$, while
the largest mass of $Z_N$ is $m_N=2^{-N}\max_z\Pr(Z=z)\to0$.  If $T_N$ is the
event of a tie among eight iid draws, then
$\Pr(T_N)\leq\binom82m_N\to0$.  Off $T_N$, $H_{Z_N}$ depends only on the
relative ranks, whose orders are equiprobable.  Hence, for $U$ uniform on
$(0,1)$, the bound $|H_Z|\leq2/3$ gives
$|d(Z_N)-d(U)|\leq(8/9)\Pr(T_N)\to0$.  Therefore
$d(Z)\leq d(U)=D(U,U)=\int_0^1u^2(1-u)^2\,du=1/30$.
For arbitrary $Z$, Lemma~\ref{lem:quantization} gives finite-support
$Z^{(k)}=q_k(Z)$ with $d(Z^{(k)})\to d(Z)$, so the bound passes to the limit.

For atomless $Z$, exactly one of $E_1^Z,E_2^Z,E_3^Z$ occurs almost surely, so
$\E A_Z^2=2/3$.  Moreover,
$\Da(Z,Z)=D(Z,Z)=1/30$.  A strictly increasing
transformation preserves every
weak order, while a strictly decreasing one reverses all orders and leaves
\eqref{eq:interval-form} unchanged.  Thus, in the equality case stated in the
proposition, $A_Y=A_X$ and $H_Y=H_X$, so both upper bounds are attained.
Finally, the calculation and sign check in Section~6 of the supplement
give \eqref{eq:Ra-uniform-range-counterexample} and show that its numerator is
positive exactly for integers $m\geq14$.
\end{proof}

Consequently, for arbitrary laws one may still write
\[
 \ts_{\rm N}=\frac32\ts,\qquad
 \Da_{\rm N}=30\Da,\qquad
 \Ra_{\rm c}=90\Ra,
\]
and the decomposition becomes
\begin{equation}\label{eq:normalised-tie-aware}
 \boxed{\qquad
 \ts_{\rm N}=\frac35\Da_{\rm N}+\frac25\Ra_{\rm c}.
 \qquad}
\end{equation}
Here $\ts_{\rm N}$ and $\Da_{\rm N}$ are genuinely maximum-normalised on the
class of all laws, whereas $\Ra_{\rm c}$ is normalised on the classical
atomless scale, satisfies $0\leq\Ra_{\rm c}\leq5/2$, and can exceed one; the
upper bound is not claimed sharp.  Thus
\eqref{eq:normalised-tie-aware} is a convex decomposition of three unit-range
coefficients only in the atomless case.

\subsection{Further structural consequences}

The component $\Da$ does not characterise independence over all laws: the
dependent Yanagimoto law used above has uniform margins and $\Da=0$, so its
dependence is carried entirely by $\Ra$.  The projection-averaged multivariate
$\ts$ of \citet[Def.~7.1]{KBW} inherits the decomposition, nonnegativity and
zero-set results; the proof is given in Section~1 of the supplement.  Section~4
details
two further consequences: each bound
$\Da\geq D_{\epsilon,\eta}/9$, $\epsilon,\eta\in\{<,\leq\}$, is sharp, and
every finite table with a binary margin satisfies $\ts=36\Ra$ and $\Da=\Ra$.

\subsection{Outlook}

We leave open the determination of the sharp global supremum of $\Ra$ over
arbitrary laws with ties.  By Cauchy--Schwarz and the choice $X=Y=Z$, it equals
$\sup_{\mathcal L(Z)}\Ra(Z,Z)$; Prop.~\ref{prop:sharp-scales} places it
in $(1/90,1/36]$.

The projection structure in Lemma~\ref{lem:pair-projection} is generic, but the
nonnegativity of both cross-components is special to the $\ts$ kernel.  A
further direction is to identify other symmetric rank covariances whose
components are nonnegative and interpretable, determine independence, and
admit explicit tie-symmetrised representations and sharp normalisations.

\section*{Acknowledgments}

The authors used OpenAI's Codex and Anthropic's Claude during proof development
and manuscript preparation.  All AI-assisted material was critically reviewed,
substantially revised, and independently verified by the authors, who take full
responsibility for the manuscript.

\section*{Supplementary material}
The supplementary results and calculations follow the references in this
preprint. Exact Python/SymPy verification scripts are supplied as arXiv
ancillary files in \texttt{anc/exact\_algebra/}.

\bibliographystyle{plainnat}
\bibliography{taustar2_refs}

\clearpage
\setcounter{section}{0}
\setcounter{equation}{0}
\setcounter{theorem}{0}
\renewcommand{\theHsection}{supp.\arabic{section}}
\renewcommand{\theHequation}{supp.\arabic{equation}}
\renewcommand{\theHtheorem}{supp.\arabic{theorem}}
\begin{center}
{\Large\bfseries Supplementary results and calculations\par}
\vspace{6pt}
{\large Wicher Bergsma and Angelos Dassios\par}
\end{center}

\begingroup\small
\begin{center}\bfseries Supplement overview\end{center}
\begin{quotation}\noindent
We extend the tie-symmetrised decomposition and its zero-set characterization
to projection-averaged functionals of random vectors.  We also prove
consistency of the permutation test, give the general quadratic
monotone-association bounds and their sharpness, record individual-boundary
sharpness and binary-margin consequences, detail the atomless kernel
identification, give the discrete-uniform calculation used for the range
counterexample, and prove the atom-refinement lemma used to obtain a sharp
global scale in the main paper.  The arXiv ancillary files contain exact-algebra verification scripts.
\end{quotation}\endgroup

The supplement is organised as follows.  Sections~\ref{sec:supp-multivariate}
and~\ref{sec:supp-permutation} give the multivariate extension and permutation
consistency.  Sections~\ref{sec:supp-quadratic} and~\ref{sec:supp-boundary}
establish the additional sharpness and binary-margin results.  The atomless,
discrete-uniform and atom-refinement calculations are given in
Sections~\ref{sec:supp-atomless}--\ref{sec:supp-refinement}, and
Section~\ref{sec:supp-exact-checks} records the exact-algebra checks.

\section{Projection-averaged multivariate extension}
\label{sec:supp-multivariate}

Let $p,q\geq1$, let $X\in\R^p$ and $Y\in\R^q$, and let $\lambda_d$ denote
uniform probability measure on $\mathbb S^{d-1}$.  Following the
projection-averaged multivariate $\ts$ of \citet[Def.~7.1]{KBW}, for
$T\in\{\ts,\Da,\Ra\}$ define
\[
 T_{p,q}(X,Y)
 =\int_{\mathbb S^{p-1}}\int_{\mathbb S^{q-1}}
 T(\alpha^{\mathsf T}X,\beta^{\mathsf T}Y)\,
 d\lambda_q(\beta)d\lambda_p(\alpha).
\]

\begin{proposition}[Projection-averaged extension]
\label{prop:supp-multivariate}
The three projection-averaged functionals are well defined and satisfy
\[
 \ts_{p,q}=12\Da_{p,q}+24\Ra_{p,q},
 \qquad \Da_{p,q},\Ra_{p,q}\geq0.
\]
Moreover,
\[
 \Ra_{p,q}=0\quad\Longleftrightarrow\quad
 \ts_{p,q}=0\quad\Longleftrightarrow\quad X\indep Y.
\]
\end{proposition}

\begin{proof}
In the order-four representation of $\ts$, the order-six representation of
$\Ra$, and the rectangular order-eight representation of $\Da$, the
sample-level integrand is a bounded Borel function jointly of
$(\alpha,\beta)$ and the finitely many sample coordinates: it is obtained
from linear forms by finitely many arithmetic operations, absolute values and
the Borel sign function.  Integrating the sample coordinates against the
appropriate product law, the measurable-parameter integral theorem then shows
that
$(\alpha,\beta)\mapsto T(\alpha^{\mathsf T}X,\beta^{\mathsf T}Y)$ is Borel
measurable, so the displayed integrals are well defined.  The scalar
decomposition in the main paper holds for every bivariate law, including laws
with ties.  Integration gives the stated decomposition.
The scalar components are nonnegative, so $\Da_{p,q},\Ra_{p,q}\geq0$ and hence
$\ts_{p,q}\geq0$, without regularity assumptions on the vector law.

Moreover, $\Ra_{p,q}=0$ forces
$\Ra(\alpha^{\mathsf T}X,\beta^{\mathsf T}Y)=0$ for almost every pair of
directions.  The scalar zero-set result then gives independence of the
corresponding projections.  This full-measure set of direction pairs is dense
because the sphere measures have full support.  For every pair in that set
and every $s,t\in\R$, independence gives
\[
 \varphi_{X,Y}(s\alpha,t\beta)
 =\varphi_X(s\alpha)\varphi_Y(t\beta).
\]
The arguments $(s\alpha,t\beta)$ so obtained are dense in
$\R^p\times\R^q$.  Continuity of characteristic functions extends the
factorisation everywhere and proves $X\indep Y$.  The converse follows
because independence is preserved by measurable transformations.  Finally,
$\ts_{p,q}=0$ forces $\Ra_{p,q}=0$ by the decomposition and nonnegativity of
both components, which completes the equivalence.
\end{proof}

\section{Consistency of the permutation test}
\label{sec:supp-permutation}

Let $(X_i,Y_i)$, $i\geq1$, be iid copies of $(X,Y)$.  For $n\geq4$, put
$(n)_4=n(n-1)(n-2)(n-3)$.  With $a$ denoting the four-argument sign kernel
from the main paper, define, for a permutation $\pi$,
\[
 t_{n,\pi}^*=\frac1{(n)_4}
 \sum_{\substack{i_1,i_2,i_3,i_4\\\mathrm{all\ distinct}}}
 a(X_{i_1},X_{i_2},X_{i_3},X_{i_4})
 a(Y_{\pi(i_1)},Y_{\pi(i_2)},Y_{\pi(i_3)},Y_{\pi(i_4)}).
\]
This is the order-four $U$-statistic with the kernel symmetrised over its four
arguments, and the observed statistic is $t_n^*=t_{n,\mathrm{id}}^*$.  Let
$\mathcal D_n=\sigma\{(X_i,Y_i):1\leq i\leq n\}$, and let $\Pi_n$ be uniform
on the permutations of $\{1,\ldots,n\}$, conditionally on $\mathcal D_n$.
For a fixed level $\alpha\in(0,1)$, define
\[
 c_{n,\alpha}
 =\inf\{c:\Pr_{\Pi_n}(t_{n,\Pi_n}^*\leq c\mid\mathcal D_n)
                 \geq1-\alpha\}.
\]
The nonrandomised upper-tail test rejects when $t_n^*>c_{n,\alpha}$.  Data ties
are handled by $\sgn(0)=0$, and equality at the critical value is not rejected;
randomisation there may instead be used to attain exact conditional level.

\begin{proposition}[Consistency of the permutation test]
\label{prop:supp-permutation}
Under independence, this test has level at most $\alpha$ for every $n\geq4$.
Under every fixed dependent alternative,
\[
 \Pr\{t_n^*>c_{n,\alpha}\}\longrightarrow1.
\]
\end{proposition}

\begin{proof}
Finite-sample validity under independence follows from permutation invariance.
The strong law for $U$-statistics gives $t_n^*\to\ts(X,Y)$ almost surely, and
the main zero-set theorem makes this limit strictly positive under every fixed
dependent alternative.

The conditional permutation mean is exactly zero: in the average over ordered
distinct $Y$-quadruples, interchanging slots 2 and 3 pairs every term with its
negative.  A transposition changes only $O(n^3)$ of the $O(n^4)$ uniformly
bounded summands.  Consequently, there is a constant $C$, independent of the
data and of $n$, such that
\[
 \sup_{\pi,\,i<j}
 \bigl|t_{n,\pi}^*-t_{n,\pi\circ(ij)}^*\bigr|\leq\frac Cn.
\]
The random-transposition Poincar\'e inequality on the symmetric group
\citep[Fact~A.2]{FilmusEtAl}, in its $n^{-1}$-times-the-transposition-sum
normalisation, therefore gives
\[
\begin{aligned}
 \operatorname{Var}_{\Pi_n}(t_{n,\Pi_n}^*\mid\mathcal D_n)
 &\leq\frac1n\sum_{i<j}
 \E_{\Pi_n}\!\left[
   \{t_{n,\Pi_n}^*-t_{n,\Pi_n\circ(ij)}^*\}^2
   \mid\mathcal D_n\right]  \\
 &\leq \frac{C^2}{2n},
\end{aligned}
\]
Since the conditional mean is zero, Chebyshev's inequality gives, uniformly
in the data,
\[
 c_{n,\alpha}\leq\frac{C}{\sqrt{2\alpha n}}\longrightarrow0.
\]
Since $t_n^*$ converges almost surely to a positive limit under every fixed
dependent alternative, the rejection probability tends to one.
\end{proof}

\section{General quadratic bounds and their sharpness}
\label{sec:supp-quadratic}

This section supplies the details for the arbitrary-law part and the sharpness
assertions of the proposition entitled ``Quadratic bounds from monotone
association'' in the main paper.  Retain the definitions of $\ts$, $\Da$ and
$\Ra$ from that paper.  Let $(X_i,Y_i)$, $i\geq1$, be independent copies of a
real-valued pair $(X,Y)$, put
$s_Z(ij)=\sgn(Z_i-Z_j)$, and define
\[
 \tau=\E\{s_X(12)s_Y(12)\},\qquad
 \rho_S=3\E\{s_X(12)s_Y(13)\},
\]
\[
 \upsilon=\frac34\E\bigl[
   \{s_X(12)-s_X(13)\}\{s_Y(12)-s_Y(13)\}\bigr].
\]
Exchangeability gives
\begin{equation}\label{eq:supp-linear-tau}
 \tau=\frac13\rho_S+\frac23\upsilon.
\end{equation}

For $\epsilon,\eta\in\{<,\leq\}$, put
\[
 \iota_{<}(u,t)=\1\{u<t\},\qquad
 \iota_{\leq}(u,t)=\1\{u\leq t\},
\]
and define
\begin{align*}
 F_{\epsilon,\eta}(x,y)
 &=\E\{\iota_\epsilon(X,x)\iota_\eta(Y,y)\},\\
 F_{X,\epsilon}(x)&=\E\{\iota_\epsilon(X,x)\},\qquad
 F_{Y,\eta}(y)=\E\{\iota_\eta(Y,y)\},\\
 \Delta_{\epsilon,\eta}(x,y)
 &=F_{\epsilon,\eta}(x,y)-F_{X,\epsilon}(x)F_{Y,\eta}(y).
\end{align*}
Writing $F$ for the joint law and $F_X\otimes F_Y$ for the product of the
margins, let
\[
 D_{\epsilon,\eta}=\int_{\R^2}\Delta_{\epsilon,\eta}^2\,dF,
 \qquad
 R_{\epsilon,\eta}=\int_{\R^2}\Delta_{\epsilon,\eta}^2\,dF_X\,dF_Y.
\]
The decomposition and four-boundary theorem in the main paper state, for every
bivariate law,
\begin{equation}\label{eq:supp-main-inputs}
 \ts=12\Da+24\Ra,\qquad
 9\Da\geq\sum_{\epsilon,\eta}D_{\epsilon,\eta},\qquad
 9\Ra\geq\sum_{\epsilon,\eta}R_{\epsilon,\eta},
\end{equation}
where here and below the sums run over $\epsilon,\eta\in\{<,\leq\}$.
We prove
\[
 \Da\geq\frac{\upsilon^2}{81},\qquad
 \Ra\geq\frac{\rho_S^2}{324},\qquad
 \ts\geq\frac{2\tau^2+(\rho_S-\tau)^2}{9},
\]
including sharpness of all three displayed constants.

\begin{proof}[Proof of the arbitrary-law bounds and sharpness]
Put
$\Delta^{\mathrm{sym}}(x,y)=
\sum_{\epsilon,\eta}\Delta_{\epsilon,\eta}(x,y)$.
The identity
$\sgn(x-X)=\iota_{<}(X,x)+\iota_{\leq}(X,x)-1$ gives
\[
 \Delta^{\mathrm{sym}}(x,y)
 =\operatorname{Cov}\{\sgn(x-X),\sgn(y-Y)\}.
\]
Integration with respect to the joint law and the product of the margins,
respectively, gives
\[
 \int_{\R^2}\Delta^{\mathrm{sym}}\,dF
 =\E\{s_X(12)s_Y(12)\}-\E\{s_X(12)s_Y(13)\}
 =\tau-\frac13\rho_S=\frac23\upsilon,
\]
and
\[
 \int_{\R^2}\Delta^{\mathrm{sym}}\,dF_X\,dF_Y
 =\E\{s_X(21)s_Y(31)\}=\frac13\rho_S.
\]
Viewing the four discrepancies as a direct-sum vector in each of these two
$L^2$ spaces, orthogonal projection onto the common constant vector gives the
exact identities
\begin{align}
 \sum_{\epsilon,\eta}D_{\epsilon,\eta}
 &=\frac{\upsilon^2}{9}
   +\sum_{\epsilon,\eta}\int_{\R^2}
      \left\{\Delta_{\epsilon,\eta}-\frac{\upsilon}{6}\right\}^2dF,
 \label{eq:supp-D-boundary-projection}\\
 \sum_{\epsilon,\eta}R_{\epsilon,\eta}
 &=\frac{\rho_S^2}{36}
   +\sum_{\epsilon,\eta}\int_{\R^2}
      \left\{\Delta_{\epsilon,\eta}-\frac{\rho_S}{12}\right\}^2
       dF_X\,dF_Y.
 \label{eq:supp-R-boundary-projection}
\end{align}
Together with \eqref{eq:supp-main-inputs}, these prove
\[
 \Da\geq\frac{\upsilon^2}{81},\qquad
 \Ra\geq\frac{\rho_S^2}{324}.
\]
Using \eqref{eq:supp-linear-tau} in the decomposition then gives
\[
 \ts\geq\frac4{27}\upsilon^2+\frac2{27}\rho_S^2
 =\frac{2\tau^2+(\rho_S-\tau)^2}{9}
 \geq\frac29\tau^2.
\]

For sharpness, let $0<\varepsilon<1/2$, set
$r_\varepsilon=(\varepsilon,1-2\varepsilon,\varepsilon)^{\mathsf T}$ and
$t=\varepsilon^3$, and form the $3\times3$ probability table, with rows and
columns indexed in increasing order by $\{1,2,3\}$,
\[
 P_{\varepsilon,t}=r_\varepsilon r_\varepsilon^{\mathsf T}
 +t\begin{pmatrix}
       -1&0&1\\
        0&0&0\\
        1&0&-1
     \end{pmatrix}.
\]
The perturbation has zero row and column sums, so both margins are
$r_\varepsilon$; the choice of $t$ makes all entries nonnegative.  Put
\[
 v=(-1,0,1)^{\mathsf T},\qquad {\boldsymbol 1}=(1,1)^{\mathsf T},
 \qquad A=(\sgn(i-j))_{i,j=1}^3.
\]
Then $P_{\varepsilon,t}=r_\varepsilon r_\varepsilon^{\mathsf T}
-tvv^{\mathsf T}$ and $Ar_\varepsilon=(1-\varepsilon)v$.  Since
$r_\varepsilon^{\mathsf T}Av=-2(1-\varepsilon)$ and
$r_\varepsilon^{\mathsf T}Ar_\varepsilon=v^{\mathsf T}Av=0$, the standard
sign-matrix formulas give
\[
 \begin{aligned}
  \tau
  &=\operatorname{tr}(P_{\varepsilon,t}^{\mathsf T}A
       P_{\varepsilon,t}A^{\mathsf T})
    =-2t(r_\varepsilon^{\mathsf T}Av)^2
    =-8t(1-\varepsilon)^2,\\
  \rho_S
  &=3(Ar_\varepsilon)^{\mathsf T}P_{\varepsilon,t}(Ar_\varepsilon)
    =-12t(1-\varepsilon)^2.
 \end{aligned}
\]
Equation~\eqref{eq:supp-linear-tau} therefore gives
$\upsilon=(3\tau-\rho_S)/2=-6t(1-\varepsilon)^2$.

For the quadratic components, the cumulative-residual matrix and the
cut-to-gap matrices of the main paper are
\[
 \mathcal X=-t{\boldsymbol 1}{\boldsymbol 1}^{\mathsf T},\qquad
 G(r_\varepsilon)=(\varepsilon,\varepsilon),\qquad
 K(r_\varepsilon)=I_2-\varepsilon^2
   {\boldsymbol 1}{\boldsymbol 1}^{\mathsf T}.
\]
Thus ${\boldsymbol 1}^{\mathsf T}K(r_\varepsilon){\boldsymbol 1}
=2(1-2\varepsilon^2)$, and the finite quadratic form gives
\[
 9\Ra=\operatorname{tr}\{K(r_\varepsilon)\mathcal X
 K(r_\varepsilon)\mathcal X^{\mathsf T}\}
 =4t^2(1-2\varepsilon^2)^2.
\]
There is only one gap correction.  In its notation,
$s_1=2\varepsilon$, $z_1=-2\varepsilon t{\boldsymbol 1}$ and
$c^{(1)}=2\varepsilon r_\varepsilon$, so
$K(c^{(1)})=4\varepsilon^2K(r_\varepsilon)$.  Hence that correction is zero,
\[
 z_1^{\mathsf T}K(r_\varepsilon)z_1
 -(2\varepsilon)^{-2}z_1^{\mathsf T}K(c^{(1)})z_1=0,
\]
and $\Da=\Ra$.  Finally, the decomposition gives
$\ts=12\Da+24\Ra=36\Ra$.  We have therefore obtained
\[
 \begin{gathered}
  \tau=-8t(1-\varepsilon)^2,\qquad
  \rho_S=-12t(1-\varepsilon)^2,\qquad
  \upsilon=-6t(1-\varepsilon)^2,\\
  \Da=\Ra=\frac{4t^2(1-2\varepsilon^2)^2}{9},\qquad
  \ts=16t^2(1-2\varepsilon^2)^2.
 \end{gathered}
\]
Consequently,
\[
 \frac{81\Da}{\upsilon^2}
 =\frac{324\Ra}{\rho_S^2}
 =\frac{9\ts}{2\tau^2+(\rho_S-\tau)^2}
 =\frac{(1-2\varepsilon^2)^2}{(1-\varepsilon)^4}
 \longrightarrow1.
\]
Thus both component constants and the first general quadratic bound are sharp.
Moreover,
\[
 \frac{\ts}{\tau^2}
 =\frac{(1-2\varepsilon^2)^2}{4(1-\varepsilon)^4}
 \longrightarrow\frac14.
\]
Hence no universal inequality $\ts\geq c\tau^2$ can have $c>1/4$, and the
stronger atomless bound $\ts\geq\tfrac12\tau^2$ cannot extend to arbitrary
laws.
\end{proof}

\section{Individual boundary sharpness and binary margins}
\label{sec:supp-boundary}

The four-boundary theorem in the main paper implies, separately for every
$\epsilon,\eta\in\{<,\leq\}$,
\[
 \Da\geq\frac19D_{\epsilon,\eta}.
\]
The factor $1/9$ is sharp for each choice.  For $0<t<1$, let
$P_t=\left(\begin{smallmatrix}1-t&0\\0&t\end{smallmatrix}\right)$ and
$\theta=t(1-t)$.  The finite-table formulas give
$D_{\leq,\leq}=(1-t)\theta^2$ and $\Da=\theta^2/9$, and hence
\[
 \frac{\Da}{D_{\leq,\leq}}
 =\frac{1}{9(1-t)}\longrightarrow\frac19
 \qquad\text{as }t\downarrow0.
\]
This proves sharpness for $(\epsilon,\eta)=(\leq,\leq)$.  Reflecting either
coordinate leaves $\Da$ unchanged and permutes the four boundary functionals,
so the same example proves sharpness for the other three choices.

For the binary-margin consequence, let $P$ be a finite $m\times n$ table on
attained row and column levels, with $\min(m,n)=2$.  The gap--gap correction
formula in the main paper gives
$9\Da(P)=9\Ra(P)+3\mathcal E(P)$, and $\mathcal E(P)=0$ throughout this
subclass.  Thus $\Da=\Ra$, and the decomposition
$\ts=12\Da+24\Ra$ yields
\[
 \ts=36\Ra,\qquad \Da=\Ra.
\]
Hence the identities used for $2\times2$ tables extend to every finite table
with a binary margin.  Equality in the $\Ra$ four-boundary bound does not
extend throughout this subclass: the equality characterisation in the main
paper shows that a dependent finite table attains it only when $m=n=2$.  No
corresponding classification of equality in the $\Da$ four-boundary bound is
asserted.

\section{Details of the atomless kernel identification}
\label{sec:supp-atomless}

Retain the definitions of $a$, $k_Z$ and $\kappa_Z$ from the main paper, and
suppose that $F_Z$ is continuous.  For
$\mathcal L(Z)^{\otimes2}$-almost every $(z,z')$, put
$u=F_Z(z)$ and $v=F_Z(z')$.  Using the second representation of $k_Z$, the
probability-integral transform and the weak-order invariance of $a$ reduce the
calculation to two independent uniform random variables $U_2,U_4$.  Suppose
first that $u<v$.  Then
$a(u,U_2,v,U_4)=1$ when both uniform variables lie below $u$ or both lie above
$v$, a set of area
\[
 u^2+(1-v)^2.
\]
It equals $-1$ on
$\{U_2<v,\ U_4>u,\ U_2<U_4\}$, whose area is
\[
 \int_0^u(1-u)\,dx+\int_u^v(1-x)\,dx
 =v-\frac{u^2+v^2}{2},
\]
and is zero elsewhere, up to null boundary sets.  Consequently,
\[
 \begin{aligned}
 k_Z(z,z')
 &=\frac13\left\{u^2+(1-v)^2-v+\frac{u^2+v^2}{2}\right\}\\
 &=\frac{u^2+v^2}{2}-v+\frac13.
 \end{aligned}
\]
Symmetry gives the same formula with $v$ replaced by $\max\{u,v\}$.

On the other hand, the centred Cram\'er--von Mises kernel becomes, for $u<v$,
\[
 \begin{aligned}
 \kappa_Z(z,z')
 &=\int_0^u t^2\,dt-\int_u^v t(1-t)\,dt
   +\int_v^1(1-t)^2\,dt\\
 &=\frac{u^2+v^2}{2}-v+\frac13.
 \end{aligned}
\]
This proves the off-diagonal identity used in the atomless kernel
identification in the main paper.
For completeness, on the diagonal $u=F_Z(z)$,
$a(u,U_2,u,U_4)=1$ precisely when both uniform variables lie on the same side
of $u$, and is zero otherwise, almost surely.  Hence
\[
 k_Z(z,z)=\frac{u^2+(1-u)^2}{3},\qquad
 k_Z(z,z)-\kappa_Z(z,z)=\frac{u(1-u)}{3}.
\]
The diagonal has zero probability for two independent draws from an atomless
law, which explains the almost-everywhere qualification in the main paper.

\section{The discrete-uniform calculation}
\label{sec:supp-discrete-uniform}

Let $Z$ be uniform on $\{1,\ldots,m\}$.  Direct summation gives
{\small
\[
 \begin{aligned}
 3m^2k_Z(i,j)
  &=(i-1)^2+(m-j)^2-i(m-i)\\
  &\quad-\tfrac12(j-i-1)(2m-i-j),\quad i<j,\\
 3m^2k_Z(i,i)&=(i-1)^2+(m-i)^2.
 \end{aligned}
\]}
Using the symmetry of $k_Z$, substitution and summation of powers yield
\[
 \Ra(Z,Z)=\frac1{m^2}\left\{\sum_{i=1}^m k_Z(i,i)^2
 +2\sum_{1\leq i<j\leq m}k_Z(i,j)^2\right\}
 =\frac1{90}+\frac{m^4-15m^3+25m^2-18m+4}{270m^5}.
\]
To verify the claimed threshold, write the numerator as
$p(m)=m^4-15m^3+25m^2-18m+4$.  If $m=14+n$ with $n\geq0$, then
\[
 p(m)=n^4+41n^3+571n^2+2838n+1908>0.
\]
For $2\leq m\leq13$, the convex quadratic $m^2-15m+25$ is at most $-1$, so
$p(m)=m^2(m^2-15m+25)-18m+4<0$; also $p(1)=-3$.  Thus $14$ is the least
positive integer for which the excess over $1/90$ is positive.

\section{Atom-refinement monotonicity}
\label{sec:supp-refinement}

For $z=(z_1,z_2,z_3,z_4)\in\R^4$, put
\[
 S_{ij\mid kl}(z)=|z_i-z_j|+|z_k-z_l|,
 \qquad
 a(z)=\sgn\{S_{12\mid34}(z)-S_{13\mid24}(z)\},
\]
where $\sgn(0)=0$, and define
\[
 \delta(z_1,z_2,z_3,z_4)
 =a(z_1,z_2,z_3,z_4)-a(z_1,z_4,z_2,z_3).
\]
For iid copies $Z_1,Z_2,\ldots$ of a real random variable $Z$, write
$a_Z(ijkl)=a(Z_i,Z_j,Z_k,Z_l)$ and
$\delta_Z(ijkl)=\delta(Z_i,Z_j,Z_k,Z_l)$, and put
\[
 H_Z=\frac1{12}\{
   \delta_Z(1234)-\delta_Z(1256)
   -\delta_Z(3478)+\delta_Z(5678)\}.
\]
Also, if $U,V$ are independent copies of $Z$, recall the pair kernel
\[
 k_Z(z,z')=-\frac13\E\{a(z,z',U,V)\}
 =\frac13\E\{a(z,U,z',V)\}.
\]
The identities proved in the main paper give
\[
 d(Z):=\E H_Z^2=\Da(Z,Z),\qquad
 \Ra(Z,Z)=\E\{k_Z(Z_1,Z_2)^2\},\qquad
 \ts(Z,Z)=\E\{a_Z(1234)^2\}.
\]

\begin{lemma}[Atom-refinement monotonicity]
\label{lem:supp-atom-refinement}
Let $Z$ have finite ordered support, and let $\widetilde Z$ be obtained by
replacing one atom of mass $a+b$ by two consecutive atoms of masses
$a,b>0$.  Then $d(\widetilde Z)\geq d(Z)$.
\end{lemma}

\begin{proof}
Let $C$ denote the atom being split, put $c=a+b$, and let $l$ and $r$ be the
masses strictly below and strictly above $C$, so $l+c+r=1$.  Denote the two
new consecutive atoms by $A<B$.  We first prove that the internal structures
of the lower and upper blocks disappear after averaging.

Let a finite law $W$ have ordered support $w_1<\cdots<w_m$, masses $p_i$, and
cumulative masses
\[
 F_0=0,\qquad F_i=\sum_{h\leq i}p_h.
\]
The order cases in the definition of $a$ give
\begin{align*}
 3k_W(w_i,w_i)
 &=F_{i-1}^2+(1-F_i)^2,\\
 3k_W(w_i,w_j)
 &=F_{i-1}^2+(1-F_j)^2-F_i(1-F_i)
   -\sum_{i<h<j}p_h(1-F_h),\qquad i<j.
\end{align*}
Indeed, for $i<j$, the two positive contributions have both auxiliary draws
strictly below $w_i$ or strictly above $w_j$.  The negative contribution has
probability
$F_i(1-F_i)+\sum_{i<h<j}p_h(1-F_h)$.

Apply these formulas to the unsplit law $Z$.  Let $L$ and $R$ denote its
lower and upper blocks.  For blocks of positive mass, define
\[
 \bar k_{st}=\E\{k_Z(Z_1,Z_2)\mid Z_1\in s,Z_2\in t\},
 \qquad s,t\in\{L,C,R\}.
\]
Summing the preceding formulas over the endpoints within $L$ and $R$ gives
\[
 (3\bar k_{st})_{s,t\in\{L,C,R\}}
 =
 \begin{pmatrix}
  (c+r)^2 & r^2-l(c+r) & -l(c+r)-cr\\
  r^2-l(c+r) & l^2+r^2 & l^2-(l+c)r\\
  -l(c+r)-cr & l^2-(l+c)r & (l+c)^2
 \end{pmatrix}.
\]
For example, the $(L,C)$ entry is the telescoping identity
\[
 \sum_{w_i<C}p_i\,3k_Z(w_i,C)
 =l\{r^2-l(c+r)\};
\]
the other entries follow from the same finite sums, with reflection for the
upper block.
Thus these block averages are exactly the $k$-values of the three-point law
with masses $(l,c,r)$; in particular, they do not involve the individual
outer-atom masses.  If $l=0$ or $r=0$, the corresponding label is deleted from
conditional expectations and block sums (equivalently, the unused conditional
means may be assigned arbitrary values), because every unconditional term
carrying that block has zero weight.

Couple the split and unsplit laws by the nondecreasing map $\chi$ that merges
$A$ and $B$ to $C$ and fixes every outer support point, so that
$\chi(\widetilde Z)$ has the law of $Z$.  For
$s,t\in\{L,A,B,R\}$ and support points $x\in s$, $y\in t$, set
\[
 q_{st}=3\{k_{\widetilde Z}(x,y)-k_Z(\chi(x),\chi(y))\}.
\]
The finite-support formulas above show that this value depends only on the
two block labels, not on the particular choices of $x$ and $y$ in $L$ or
$R$, and give the symmetric table
\[
\begin{array}{c|rrrr}
 q_{st}&L&A&B&R\\ \hline
 L&0&b(b+2r)&-a(b+r)&-ab\\
 A&b(b+2r)&b(b+2r)&-(l+a)(b+r)&-b(l+a)\\
 B&-a(b+r)&-(l+a)(b+r)&a(2l+a)&a(2l+a)\\
 R&-ab&-b(l+a)&a(2l+a)&0
\end{array}.
\]
Let $p_L=l,p_A=a,p_B=b,p_R=r$, and write
$\bar L=L$, $\bar A=\bar B=C$, and $\bar R=R$.  Expanding the square in the
representation of $\Ra(Z,Z)$ and then using the block-average table gives
\[
\begin{aligned}
 &\Ra(\widetilde Z,\widetilde Z)-\Ra(Z,Z)\\
 &\qquad={}
 \sum_{s,t\in\{L,A,B,R\}}p_sp_t
 \left\{\frac23q_{st}\bar k_{\bar s\bar t}
              +\frac19q_{st}^2\right\}.
\end{aligned}
\]
With the zero-mass convention just stated, this formula also covers $l=0$ and
$r=0$.  It proves that the change in $\Ra(Z,Z)$ depends on all
outer atoms only through $l$ and $r$.

The same conclusion for $\ts(Z,Z)$ follows from a local diagonal formula.
For the finite law $W$ above, the four strict-separation cases for $a^2$ and
exchangeability give
\[
 \ts(W,W)
 =4\sum_{i=1}^m\{F_i^2-F_{i-1}^2\}(1-F_i)^2.
\]
Splitting $C$ changes only its summand, replacing it by the two summands for
$A$ and $B$.  Explicitly,
\begin{align*}
 \ts(\widetilde Z,\widetilde Z)-\ts(Z,Z)
 =4\big[&\{(l+a)^2-l^2\}(b+r)^2
       +\{(l+a+b)^2-(l+a)^2\}r^2\\
       &-\{(l+a+b)^2-l^2\}r^2\big],
\end{align*}
which again involves only $l,a,b,r$.  Finally, the diagonal form of the main
decomposition is
\[
 d(W)=\frac{\ts(W,W)-24\Ra(W,W)}{12}.
\]
This proves analytically that
$d(\widetilde Z)-d(Z)$ depends on the outer blocks only through their total
masses $l$ and $r$, for arbitrary finite outer support.

It remains only to collect the resulting four-level configurations.  Let
$\mathcal S=\{L,A,B,R\}$, assign split levels
\[
 v_L=0,\qquad v_A=1,\qquad v_B=2,\qquad v_R=3,
\]
and collapsed levels
\[
 v_L^0=0,\qquad v_A^0=v_B^0=1,\qquad v_R^0=2.
\]
For $\boldsymbol s=(s_1,\ldots,s_8)\in\mathcal S^8$, let
\begin{align*}
 \Gamma(\boldsymbol s)={}&
 \delta(v_{s_1},v_{s_2},v_{s_3},v_{s_4})
 -\delta(v_{s_1},v_{s_2},v_{s_5},v_{s_6})\\
 &-\delta(v_{s_3},v_{s_4},v_{s_7},v_{s_8})
 +\delta(v_{s_5},v_{s_6},v_{s_7},v_{s_8}),
\end{align*}
and define $\Gamma_0(\boldsymbol s)$ by replacing every $v_s$ with $v_s^0$.
Thus $\Gamma$ and $\Gamma_0$ are the numerators $12H$ for the split and collapsed
configurations.  Direct collection of the finite sum gives the homogeneous
identity
\begin{align*}
 144\{d(\widetilde Z)-d(Z)\}
 &={}
 \sum_{\boldsymbol s\in\mathcal S^8}
 \{\Gamma(\boldsymbol s)^2-\Gamma_0(\boldsymbol s)^2\}
 \prod_{j=1}^8p_{s_j}\\
 &=16ab\{9Q(l,a,b,r)\}(l+a+b+r)^2.
\end{align*}
Since $l+a+b+r=1$, this is
$d(\widetilde Z)-d(Z)=ab\,Q(l,a,b,r)$, where
\begin{align*}
9Q(l,a,b,r)={}&(6l^3+7al^2+4a^2l+a^3)(b+2r)
 +4l^2(2b^2+6br+7r^2)\\
&+2a(3l+a)(b^2+3br+4r^2)
 +(2l+a)(b^3+4b^2r+7br^2+6r^3).
\end{align*}
Only weak-order configurations containing observations from both new subatoms
can contribute to the difference, accounting for the factor $ab$.  Every
factor and coefficient in the displayed sum is nonnegative, proving the claim.

\end{proof}

\section{Exact-algebra verification}
\label{sec:supp-exact-checks}

The script \path{anc/exact_algebra/verify_atom_refinement.py} rebuilds $a$,
$\delta$ and $\Gamma=12H$, performs the $4^8$ configuration sum in
Section~\ref{sec:supp-refinement}, and verifies both the homogeneous identity
and its nonnegative factorisation.  Its \texttt{--deep} option repeats the
check with two distinct support points in each outer block.  This is an
additional two-level check; the analytic block-aggregation argument in that
section is what covers arbitrary finite outer support.

The separate exact-arithmetic script
\path{anc/exact_algebra/verify_local_decomposition.py}
verifies the coefficient matching in the local certificate from
Proposition~10 of the main paper,
\[
 \mathcal B_a+\mathcal C_a=\mathcal N_a+
 \left(\frac1\ell+\frac1{m_0}+\frac1u-6\right)\|\xi_a\|_2^2,
\]
and the stronger exact sum-of-squares certificate showing that the displayed
coefficient is at least $3$.

\end{document}